\documentclass{amsart}
\usepackage{amsmath}
\usepackage{amsfonts}
\usepackage{amsthm}
\usepackage{amssymb}
\usepackage{url}
\usepackage{cite}

\newtheorem{thm}{Theorem}[section]
\newtheorem*{theorem*}{Theorem}
\newtheorem*{acknowledgement*}{Acknowledgements}
\newtheorem{lem}[thm]{Lemma}
\newtheorem{prop}[thm]{Proposition}
\theoremstyle{remark}
\newtheorem{rem}[thm]{Remark}
\numberwithin{equation}{section}

\newcommand{\nN}{\mathbf{n}}
\newcommand{\NN}{\mathbf{N}}
\newcommand{\pP}{\mathbf{p}}
\newcommand{\qQ}{\mathbf{q}}
\newcommand{\Real}{\mathbb{R}}
\newcommand{\set}[1]{\left\{ #1\right\}}
\newcommand{\vV}{\mathbf{v}}
\newcommand{\wW}{\mathbf{w}}
\newcommand{\xX}{\mathbf{x}}
\newcommand{\yY}{\mathbf{y}}
\newcommand{\zZ}{\mathbf{z}}
\newcommand{\oO}{\mathbf{0}}

\title{Asymptotic Structure of Self-shrinkers}
\author{Lu Wang}
\address{Department of Mathematics, University of Wisconsin-Madison, 480 Lincoln Drive, Madison, WI 53706.}
\email{luwang@math.wisc.edu}
\thanks{The author was partly supported by the AMS-Simons Travel Grant, the Chapman Fellowship of the Imperial College London, the NSF grant DMS-1406240, and the Alfred P. Sloan Research Fellowship. This material is based upon work supported by the NSF grant DMS-1440140 while the author was in residence at the Mathematical Sciences Research Institute (MSRI) in Berkeley, CA, during the Spring 2016 semester.}

\begin{document}
\begin{abstract}
We show that each end of a noncompact self-shrinker in $\mathbb{R}^3$ of finite topology is smoothly asymptotic to either a regular cone or a self-shrinking round cylinder. 
\end{abstract}
\maketitle

\section{Introduction} \label{Intro}
Given an open set $U$ of $\Real^{n+1}$, a \emph{hypersurface in $U$} is a smooth properly embedded codimension-one submanifold of $U$. A \emph{self-shrinker in $U$} is a hypersurface $\Sigma$ in $U$ that satisfies
\begin{equation}\label{ShrinkerEqn}
\mathbf{H}_\Sigma+\frac{\xX^\perp}{2}=\oO,
\end{equation}
where $\mathbf{H}_\Sigma=-H_\Sigma\nN_\Sigma=\Delta_\Sigma\xX$ is the mean curvature vector of $\Sigma$ and $\xX^\perp$ is the normal part of the position vector. Our definitions ensure that self-shrinkers in $\Real^{n+1}$ are geodesically complete. 
Self-shrinkers generate solutions to the mean curvature flow that move self-similarly by scaling. That is, if $\Sigma$ is a self-shrinker in $\mathbb{R}^{n+1}$, then the family $\{\sqrt{-t}\, \Sigma\}_{t<0}$ flows by mean curvature. Self-shrinkers play an important role in the study of mean curvature flow, not least because they are models for type-I singularities of the flow; cf. \cite{H1,H2}. Important examples of self-shrinkers in $\Real^{n+1}$ are cylindrical products $\Real^{n-k}\times\mathbb{S}^k$ ($0\leq k\leq n$) where $\mathbb{S}^k$ is the $k$-sphere (centered at the origin) of radius $\sqrt{2k}$.

In this paper we give a geometric description of the asymptotic structure of self-shrinkers in $\Real^3$ of finite topology, confirming a conjecture of Ilmanen \cite[p. 39]{IlmanenLec}.
Given $\vV\in\Real^3\setminus\{\oO\}$, let $\Real_\vV$ be the subspace of $\Real^3$ spanned by $\vV$, and we will simply write $\Real$ when the generator is not specified. Then
$$
\Real_\vV\times\mathbb{S}^1=\set{\xX\in\Real^3\colon\mathrm{dist}(\xX,\Real_\vV)=\sqrt{2}}.
$$
We call $\mathcal{C}$ a \emph{regular cone} in $\Real^3$ if there is a smooth embedded codimension-one submanifold $\mathcal{L}$ of the unit $2$-sphere so that $\mathcal{C}=\set{\tau\mathcal{L}\colon\tau>0}$. 
\begin{thm} \label{MainThm}
If $M$ is an end of a noncompact self-shrinker in $\mathbb{R}^3$ of finite topology, then either of the following holds:
\begin{enumerate}
\item \label{RegCone} $\lim_{\tau\to +\infty}\tau^{-1} M= \mathcal{C}(M)$ in $C^\infty_{loc} (\Real^3\setminus\{\oO\})$ for $\mathcal{C}(M)$ a regular cone in $\Real^3$.
\item \label{Cylinder} $\lim_{\tau\to +\infty} (M-\tau\vV(M))=\Real_{\vV(M)}\times\mathbb{S}^1$ in $C^\infty_{loc} (\Real^3)$ for a $\vV(M)\in\Real^3\setminus\{\oO\}$.
\end{enumerate}
\end{thm}
\begin{rem} 
Theorem \ref{MainThm} applies to all singularity models of finite topology at first singular time of the mean curvature flow in $\Real^3$ starting from a surface with finite genus and bounded area ratios (cf. \cite{IlmanenSing}).
\end{rem}

The classification problem for self-shrinkers is a central topic in the study of singularities of mean curvature flow. There are many significant results on characterizing the simplest self-shrinkers $\Real^{n-k}\times\mathbb{S}^k$; cf. \cite{Be}, \cite{CIM}, \cite{CMGenSing}, \cite{EHGraph}, \cite{H1}, \cite{H2}, \cite{WaBern}. However, the space of self-shrinkers of general dimensions is very large and complicated. Indeed, there exist rich families of self-shrinkers in $\Real^3$ with various geometric and topological properties that are constructed by the shooting method, the gluing techniques, or the min-max method; cf. \cite{A}, \cite{KKM}, \cite{K}, \cite{N}. In variational point of view, self-shrinkers are minimal with respect to a Gaussian-like conformal change of the Euclidean metric on $\Real^{n+1}$ which cannot be extended to a complete one and has scalar curvature becoming negative and unbounded near infinity. As such, we are particularly interested in noncompact self-shrinkers where the classical theory of minimal submanifolds does not apply.

Theorem \ref{MainThm} serves an important step towards a thorough understanding of the nature of noncompact self-shrinkers in $\Real^3$ of finite topology. 
Using Carleman type techniques we proved, in \cite{WaCone}, that two self-shrinkers in $\Real^{3}$ smoothly asymptotic along some end of each to the same cone must identically coincide with each other. There is also a conjecture of Ilmanen \cite[p. 39]{IlmanenLec} concerning the uniqueness of $\Real\times\mathbb{S}^1$ among all self-shrinkers in $\Real^3$ with an asymptotically cylindrical end. In \cite{WaCyl} we confirmed this conjecture under a fast decay rate condition. In view of these and Theorem \ref{MainThm}, the classification of noncompact self-shrinkers in $\Real^3$ of finite topology is reduced to the complete resolution of Ilmanen's uniqueness conjecture and the characterization of the asymptotic cones of self-shrinkers (cf. \cite{CS}).

It may be interesting to compare Theorem \ref{MainThm} with results on the asymptotic structure of (properly embedded) minimal surfaces in $\Real^3$ of finite topology. We divide our discussion into two situations according to the total curvature (i.e., the $L^2$-norm of the length of second fundamental form) finite or not. First, by the Choi-Schoen curvature estimate \cite{CS}, each end of a minimal surface in $\Real^3$ that has finite total curvature is asymptotically flat. Likewise, by Ecker's partial regularity theorem \cite{Ec}, the same conclusion stands if replacing minimal surfaces with self-shrinkers. However, if an end of a surface in $\Real^3$ of finite topology has infinite total curvature, then, in the minimal case, by work of Colding-Minicozzi \cite{CMMS1,CMMS2,CMMS3,CMMS4} and Meeks-Rosenberg \cite{MR}, it looks roughly like a helicoid near infinity, which exhibits a sharp contrast to the self-shrinking case. 

In what follows we give an outline of the proof of Theorem \ref{MainThm}. For $M$ we first introduce the associated Brakke flow $\mathcal{K}$ in an open connected set in space-time; see \eqref{ShrinkerFlowEqn}. Let $\Lambda$ be the set of all points $\yY\in\Real^3\setminus\set{\oO}$ such that $\Theta_{(\yY,0)}(\mathcal{K})$ -- the Gaussian density of $\mathcal{K}$ at $(\yY,0)$ (cf. item \eqref{GaussDensity} of Proposition \ref{WhiteProp}) -- is larger than or equal to $1$.
We show that $\Lambda$ is a cone (i.e., a dilation invariant set) and its link is nonempty compact. Furthermore, as $t\to 0^-$, $\sqrt{-t}\, \bar{M}$ converges to $\bar{\Lambda}$ locally in the Hausdorff metric; see Proposition \ref{HausdorffAsympProp}.

To prove Theorem \ref{MainThm} we need to show the regularity of $\Lambda$ and the more refined converging process. To achieve these, it is natural to use blowup analysis to study the asymptotic behaviors of the flow $\mathcal{K}$ near time $0$. Namely, take $\yY\in\Lambda$, and consider parabolic rescalings $\mathcal{K}^{(\yY,0),\tau}$ of $\mathcal{K}$ about $(\yY,0)$; cf. \eqref{ScalingShrinkerEqn}. Observe that the $-1$ time slice of $\mathcal{K}^{(\yY,0),\tau}$ is $\mathcal{H}^2\lfloor (M-\tau\yY)$ where $\mathcal{H}^2$ is the $2$-dimensional Hausdorff measure. The Brakke flow given by the limit of a sequence of $\mathcal{K}^{(\yY,0),\bar{\tau}_i}$ as $\bar{\tau}_i\to+\infty$ is called a \emph{tangent flow to $\mathcal{K}$ at $(\yY,0)$}.
The work of White \cite{WhStrata} ensures the existence of tangent flows and that they are backwardly self-similar about $(\oO,0)$; see item \eqref{TangentFlow} of Proposition \ref{WhiteProp}. Moreover, by the proof of \cite[Theorem 2]{IlmanenSing}, each tangent flow to $\mathcal{K}$ at $(\yY,0)$ has smooth support given by a self-shrinker in $\Real^3$ but possibly high multiplicities. Furthermore, by White's stratification theorem \cite{WhStrata} and the classification theorem of Abresh-Langer \cite{AL}, the supporting self-shrinker must be $\Real_\yY\times\Real$ or $\Real_\yY\times\mathbb{S}^1$. We refer the reader to Proposition \ref{SplitProp} for details. Thus, by Brakke's regularity theorem \cite{B} (cf. \cite{WhReg}), Theorem \ref{MainThm} would follow if one can show the multiplicity equal to $1$.

Indeed, Ilmanen conjectured that every tangent flow at first singular time of the mean curvature flow in $\Real^3$ starting from a surface with finite genus and bounded area ratios must have multiplicity one; cf. \cite[p. 7]{IlmanenSing} and \cite[p. 39]{IlmanenLec}. However, this conjecture is wide open except under the compact mean convex assumption, \cite{WhSize}, or the Andrews condition, \cite{HK}, or the bounded mean curvature assumption, \cite{LW}, which may be invalid in our situation. Thus we develop a strategy from scratch, independent of those previously mentioned papers,  to address this conjecture for the flow $\mathcal{K}$ in question. The core of the proof is the following sheeting theorem. Through the paper let $B^n_R(\xX)$ be the open Euclidean ball in $\Real^n$ centered at $\xX$ with radius $R$, and we will omit the center if it is the origin.
\begin{thm}\label{SheetThm}
Let $M$ be an end of a noncompact self-shrinker in $\Real^3$ of finite topology, and let $\mathcal{K}$ be the associated Brakke flow for $M$. Given $\yY\in\Real^3\setminus\{\oO\}$ with $\Theta_{(\yY,0)}(\mathcal{K})\geq 1$, there exist
\begin{itemize} 
\item a positive integer $L$, a large real number $\tau_0$,
\item an increasing unbounded positive function $R(\tau)$ for $\tau>\tau_0$, and,
\item a collection $\set{\Sigma^\tau}_{\tau>\tau_0}$ of self-shrinkers in $\Real^3$ with the property that either the $\Sigma^\tau$ are all $\Real_\yY\times\mathbb{S}^1$, or they are all of the form $\Real_\yY\times\Real$,
\end{itemize}
depending only on $M$ and $\yY$, such that for each $\tau>\tau_0$,
\begin{enumerate}
\item \label{Component} $(M-\tau\yY)\cap B^3_{R(\tau)}$ decomposes into $L$ connected components, $M^\tau_1,\ldots, M^\tau_L$; 
\item \label{Sheet} each $M^\tau_j$ is given by the normal exponential graph of a function $f_j^\tau$ over some subset $\Omega_j^\tau\subset \Sigma^\tau$. 
\end{enumerate}
Furthermore, for each $j$, 
\begin{equation} \label{C1EstEqn}
\lim_{\tau\to +\infty}\sup_{\Omega_j^\tau} |f_j^\tau|+|\nabla_{\Sigma^\tau} f_j^\tau|=0.
\end{equation}
\end{thm}

To prove Theorem \ref{SheetThm} we will need the following graphical property for self-shrinkers in unit balls with centers away from the origin under certain hypotheses. 
Let $A_\Sigma$ be the second fundamental form of $\Sigma$. 
\begin{thm} \label{GraphThm}
Given $\kappa>0$ and $\delta\in (0,1)$ there exist constants $0<\alpha, \rho<1$ and $\mathcal{R}>0$ depending only on $\kappa$ and $\delta$ such that given $\xX_0\in\Real^3\setminus\bar{B}^3_{\mathcal{R}}$ and a self-shrinker $\Sigma\subset B^3_1(\xX_0)$ with $\xX_0\in\Sigma$, if
\begin{align} 
 \label{SlcurvEqn} \sup_{\xX\in\Sigma} |\xX|^{-1} |A_\Sigma(\xX)| & <\alpha, \mbox{ and} \\
\label{TotalcurvEqn}  \int_\Sigma |A_\Sigma|^2\, d\mathcal{H}^2 & <\kappa, 
\end{align}
then the connected component of $\Sigma\cap B^3_{\rho}(\xX_0)$ that contains $\xX_0$ can be written as the graph of a function over some subset of $T_{\xX_0}\Sigma$ with its gradient bounded by $\delta$.
\end{thm}
\begin{rem}
It is unknown that whether the hessian of the function in Theorem \ref{GraphThm} is uniformly bounded independent of $\xX_0$ or not. This does not seem to follow directly from standard regularity theorems, such as Allard's regularity theorem or the Schauder estimates.
\end{rem}

Together with a local version of Huisken's monotonicity formula \cite[Section 10]{WhStrata}, the local Gauss-Bonnet estimate \cite[Theorem 3]{IlmanenSing}, and the linear growth bound on second fundamental form \cite[Theorem 19]{Song}, we first apply Theorem \ref{GraphThm} to show that there exist two sequences $R_i \nearrow +\infty$, $\tau_i\nearrow +\infty$ with $\tau_{i+1}\tau_i^{-1}\to 1$ and a sequence $\Sigma^{\tau_i}$ of self-shrinkers in $\Real^3$ such that for each $i$, items \eqref{Component}, \eqref{Sheet} of Theorem \ref{SheetThm} hold by replacing $\tau$ by $\tau_i$, $R(\tau)$ by $2R_i$, and $\Sigma^\tau$ by $\Sigma^{\tau_i}$. We further derive estimates similar to \eqref{C1EstEqn}; see Proposition \ref{SeqSheetProp}.
If $\bigcup_i B^3_{R_i}(\tau_i\yY)$ covers $I_\yY(a,+\infty)$, the set of all points $\xX\in\Real_{\yY}$ such that $\xX\cdot\yY>a|\yY|$, for some $a>0$, then we are done. Otherwise, we argue by contradiction to show that for all sufficiently large $i$, the connected components of $M\cap B^3_{R_i}(\tau_i\yY)$ extend to infinity and remain disjoint in $I_{\yY}(\tau_i|\yY|,+\infty)\times B^2_{R_i}$. Namely, assume the negation of the claim. We use the local monotonicity formula \cite[Section 10]{WhStrata} and a special case of the maximum principle \cite[Proposition 6]{Song} (cf. Lemma \ref{MaxPrincipleLem}) to show the existence of simple closed curves in $M$ that do not separate $M$. However, by the hypothesis, we may assume that $M$ is homeomorphic to an open annulus. This gives a contradiction. Combining these with Brakke's regularity theorem \cite{B} (cf. \cite{WhReg}) finishes the proof of Theorem \ref{SheetThm}.

Finally we show $L=1$; that is, in view of \eqref{ScalingShrinkerEqn}, the multiplicity of each tangent flow to $\mathcal{K}$ at $(\yY,0)$ is one. If the $\Sigma^\tau$ are all $\Real_\yY\times\mathbb{S}^1$, then the claim follows from that $M$ is connected. Otherwise, we generalize White's arguments \cite[Section 7]{Ch} for mean curvature flow of curves to higher dimensions to prove the claim.

\section{Hausdorff convergence to asymptotic cones}
In this section we study the asymptotic behaviors (in the Hausdorff metric) of noncompact self-shrinkers of general dimensions. 

\subsection{Weak mean curvature flows}
Let $\mathbf{T}\colon\Real^{n+1}\times\Real\to\Real$ denote the projection onto the time axis given by $\mathbf{T}(\xX,t)=t$. Let $X$ be a connected open subset of $\Real^{n+1}\times\Real$ in space-time. Denote by $X_t=\set{\xX\in\Real^{n+1} \colon (\xX,t)\in X}$. As $X$ is open connected, $\mathbf{T}(X)$ is an open interval in $\Real$. 
A \emph{Brakke flow in $X$} is a family $\mathcal{K}=\set{\mu_t}_{t\in \mathbf{T}(X)}$, where each $\mu_t$ is a Radon measure on $X_t$, satisfying that there is a dense set $J\subset \mathbf{T}(X)$ so that if $t\in J$, then $\mu_t$ is $n$-rectifiable and for all $\phi\in C^2_c(\Real^{n+1}; \Real^{\ge 0})$,
$$
\bar{\partial}_t \int \phi\, d\mu_t \leq \int -\phi^2 |\mathbf{H}|^2+(D\phi)^\perp\cdot\mathbf{H}\, d\mu_t
$$
where $\bar{\partial}_t$ denotes the upper derivative, $(D\phi)^\perp$ is the normal part of $D\phi$ to the approximate tangent plane, and $\mathbf{H}$ is the generalized mean curvature vector; otherwise, for all $\phi\in C^2_c(\Real^{n+1}; \Real^{\ge 0})$,
$$
\bar{\partial}_t \int\phi\, d\mu_t=-\infty.
$$
This agrees with the standard definition of Brakke flows when $X=\Real^{n+1}\times (a,b)$; cf. \cite[Section 6]{IlmanenElliptic}. A Brakke flow in $X$, $\mathcal{K}=\set{\mu_t}_{t\in \mathbf{T}(X)}$, is \emph{integral} if $\mu_t$ is integer $n$-rectifiable for a.e. $t\in \mathbf{T}(X)$.

For a set $\Omega\subset\Real^{n+1}$, an $\xX\in\Real^{n+1}$, and a $\tau>0$, let
\begin{enumerate}
\item $\Omega+\xX=\set{\yY\in\Real^{n+1} \colon \yY-\xX\in\Omega}$, the translation of $\Omega$ by $\xX$; and;
\item $\tau\Omega=\set{\yY\in\Real^{n+1} \colon \tau^{-1}\yY\in\Omega}$, the scaling of $\Omega$ by $\tau$.
\end{enumerate}
Likewise, for a set $\Omega\subset\Real^{n+1}\times\Real$ in space-time, an $(\xX,t)\in\Real^{n+1}\times\Real$, and a $\tau>0$, let
\begin{enumerate}
\item $\Omega+(\xX,t)=\set{(\yY,s)\in\Real^{n+1}\times\Real \colon (\yY-\xX,s-t)\in \Omega}$, the space-time translation of $\Omega$ by $(\xX,t)$; and;
\item $\tau \Omega=\set{(\yY,s)\in\Real^{n+1}\times\Real \colon (\tau^{-1}\yY,\tau^{-2}s)\in \Omega}$, the parabolic scaling of $\Omega$ by $\tau$.
\end{enumerate}

Given $(\yY,\tau)\in\Real^{n+1}\times\Real^+$ and $\mu$ an integer $n$-rectifiable Radon measure, define the rescaled measure $\mu^{\yY,\tau}$ by
$$
\mu^{\yY,\tau}(\Omega)=\tau^n \mu(\tau^{-1} \Omega+\yY).
$$
For a Brakke flow $\mathcal{K}=\set{\mu_t}_{t\in \mathbf{T}(X)}$ in $X$, a point $(\yY,s)\in X$, and a $\tau>0$, we define the parabolically rescaled flow in $\hat{X}=\tau (X-(\yY,s))$, 
$$
\mathcal{K}^{(\yY,s),\tau}=\set{\mu^{(\yY,s),\tau}_t}_{t\in \mathbf{T}(\hat{X})}
$$
where 
$$
\mu^{(\yY,s),\tau}_t=\mu_{s+\tau^{-2}t}^{\yY,\tau}.
$$

We summarize the results of \cite[Section 10]{WhStrata} that are needed for our purpose. The expert should feel free to skip the following proposition.
\begin{prop} \label{WhiteProp}
Let $\mathcal{K}=\set{\mu_t}_{t\in \mathbf{T}(X)}$ be an integral Brakke flow in a connected open set $X\subset\Real^{n+1}\times\Real$ in space-time. Suppose
\begin{equation} \label{ArearatioEqn}
\sup_{t\in \mathbf{T}(X)} \sup_{(\xX,R)\in\Real^{n+1}\times\Real^+} R^{-n} \mu_t(X_t\cap B^{n+1}_R(\xX))<+\infty.
\end{equation}
Then the following properties hold.
\begin{enumerate}
\item \label{GaussDensity} Given $(\yY,s)\in X$,
$$
\lim_{t \to s^-} (4\pi (s-t))^{-\frac{n}{2}} \int \phi_{\yY}(\xX) {\rm e}^{\frac{|\xX-\yY|^2}{4(t-s)}}\, d\mu_t(\xX)
$$
exists, where $\phi_{\yY}\colon \Real^{n+1}\to [0,1]$ is a smooth cutoff function about $\yY$. Moreover, the limit is independent of the choice of cutoff functions, and it is called the Gaussian density of $\mathcal{K}$ at $(\yY,s)$ denoted by $\Theta_{(\yY,s)} (\mathcal{K})$.
\item The function on $X$ defined by
$$
(\yY,s)\mapsto\Theta_{(\yY,s)} (\mathcal{K})
$$
is upper semi-continuous.
\item \label{TangentFlow} Given a point $(\yY,s)\in X$ with $\Theta_{(\yY,s)}(\mathcal{K}) \ge 1$ and a sequence $\bar{\tau}_i\to +\infty$, there exists a subsequence $\bar{\tau}_{i_k}$ and an integral Brakke flow $\mathcal{T}=\set{\nu_t}_{t\in\Real}$ in $\Real^{n+1}\times\Real$ such that $\mathcal{K}^{(\yY,0),\bar{\tau}_{i_k}}\to\mathcal{T}$, that is, $\mu_t^{(\yY,s),\bar{\tau}_{i_k}}\to\nu_t$ in the sense of Radon measures for all $t\in\Real$. The flow $\mathcal{T}$ is said to be a tangent flow to $\mathcal{K}$ at $(\yY,s)$.

Furthermore, $\mathcal{T}$ is backwardly self-similar about $(\oO,0)$. That is, for $t<0$, $\nu_t=\nu_t^{(\oO,0),\tau}$ for all $\tau>0$; $\nu_{-1}$ satisfies that
$$
\mathbf{H}+\frac{\xX^\perp}{2}=0 \mbox{ for $\nu_{-1}$-a.e. $\xX$};
$$
and $\Theta_{(\oO,0)}(\mathcal{T})=\Theta_{(\yY,s)} (\mathcal{K})$.
\end{enumerate}
\end{prop}

\subsection{Asymptotic properties for self-shrinking ends}
For an end $M$ of a noncompact self-shrinker in $\Real^{n+1}$, define the \emph{associated Brakke flow} $\mathcal{K}=\set{\mu_t}_{t\in\Real}$ by
\begin{equation} \label{ShrinkerFlowEqn}
\mu_t=\left\{
\begin{array}{cc}
0 & t\geq 0 \\
\mathcal{H}^n\lfloor \sqrt{-t}\, M & t<0,
\end{array}
\right.
\end{equation}
where $\mathcal{H}^n$ is the $n$-dimensional Hausdorff measure. 
One can verify that $\mathcal{K}$ is an integral Brakke flow in the open connected set 
$$
X=\Real^{n+1}\times\Real\setminus\bigcup_{t\leq 0} (\sqrt{-t}\cdot\partial M)\times\set{t}.
$$
Note that $X_0=\Real^{n+1}\setminus\set{\oO}$. For $\yY\in\Real^{n+1}\setminus\set{\oO}$ and $\tau>0$ the parabolic rescaling of $\mathcal{K}$ about $(\yY,0)$ by $\tau$ is
$$
\mathcal{K}^{(\yY,0),\tau}=\set{\mu_t^{(\yY,0),\tau}}_{t\in\Real}
$$
where
\begin{equation} \label{ScalingShrinkerEqn}
\mu_t^{(\yY,0),\tau}=\left\{
\begin{array}{cc}
\mathcal{H}^n\lfloor (\sqrt{-t}\, M-\tau\yY) & t<0 \\
0 & t\geq 0.
\end{array}
\right.
\end{equation}
By \cite[Theorem 1.3]{CZ} (cf. \cite[Theorem 1.1]{DX}), $\mathcal{K}$ satisfies \eqref{ArearatioEqn}. Thus, Proposition \ref{WhiteProp} applies to $\mathcal{K}$.

For a set $\Omega\subset\Real^{n+1}$ and $R>0$, denote by
$$
\mathcal{N}_R(\Omega)=\bigcup_{\xX\in\Omega} \bar{B}^{n+1}_R(\xX).
$$
Given compact sets $K,K^\prime\subset\Real^{n+1}$ the \emph{Hausdorff distance} between $K$ and $K^\prime$ is 
$$
\mathrm{dist}_H(K,K^\prime)=\inf\set{R>0 \colon K^\prime\subset\mathcal{N}_R(K), K\subset\mathcal{N}_R(K^\prime)}.
$$
It is known that $\mathrm{dist}_H$ induces a metric on the space of all compact sets of $\Real^{n+1}$ which is called the \emph{Hausdorff metric}.

We prove a claim of Ilmanen \cite[footnote on p. 8]{IlmanenLec} -- see \cite[Theorem 12]{Song} for an alternative proof.
\begin{prop} \label{HausdorffAsympProp}
Let $M$ be an end of a noncompact self-shrinker in $\Real^{n+1}$, and let $\mathcal{K}=\set{\mu_t}_{t\in\Real}$ be the associated Brakke flow for $M$. Then
\begin{enumerate}
\item  \label{AsympCone}  The set
$$
\Lambda=\set{\yY\in\Real^{n+1}\setminus\set{\oO} \colon \Theta_{(\yY,0)}(\mathcal{K})\geq 1}
$$
is a cone (i.e., $\tau\Lambda=\Lambda$ for all $\tau>0$), and the link of the cone, $\Lambda\cap\mathbb{S}^n$, is nonempty compact. 
\item \label{HausdorffAsymp} For all $R>0$, as $t\to 0^-$, $(\sqrt{-t}\, \bar{M})\cap\bar{B}^{n+1}_R\to\bar{\Lambda}\cap\bar{B}^{n+1}_R$ in the Hausdorff metric. 
\end{enumerate}
Thus the cone $\Lambda$ is called the asymptotic cone of $M$.
\end{prop}
\begin{proof}
For any $\yY\in\Real^{n+1}\setminus\set{\oO}$ and $\tau>0$,
\begin{align*}
\Theta_{(\tau\yY,0)} (\mathcal{K}) & = \lim_{t\to 0^-} (-4\pi t)^{-\frac{n}{2}} \int\phi_{\tau\yY}(\xX) {\rm e}^{\frac{|\xX-\tau\yY|^2}{4t}} \, d\mu_t(\xX) \\
& = \lim_{t\to 0^-} (-4\pi\tau^{-2}t)^{-\frac{n}{2}} \int\phi_{\tau\yY}(\tau\xX) {\rm e}^{\frac{|\xX-\yY|^2}{4\tau^{-2}t}} \, d\mu_{\tau^{-2}t}(\xX),
\end{align*}
where we used, in the last inequality, that $\phi_{\tau\yY}(\tau\cdot)$ is a cutoff function about $\yY$ and the Gaussian density is independent of the choice of cutoff functions. Thus, for every $\tau>0$, $\yY\in\Lambda$ iff $\tau\yY\in\Lambda$, implying $\Lambda=\tau\Lambda$. Moreover, by the upper semi-continuity of Gaussian density, the set $\Lambda\cap\mathbb{S}^n$ is nonempty compact. Hence we have established item \eqref{AsympCone}.

For item \eqref{HausdorffAsymp}, let us fix any $R>0$. Denote by $M_t^R=(\sqrt{-t}\, \bar{M})\cap\bar{B}^{n+1}_R$ and by $\Lambda^R=\bar{\Lambda}\cap\bar{B}^{n+1}_R$. By item \eqref{AsympCone}, $\bar{\Lambda}=\Lambda\cup\set{\oO}$. We argue by contradiction. Suppose that $M^R_t\not\to\Lambda^R$ in the Hausdorff metric as $t\to 0^-$. There was an $0<\epsilon\ll R$ and a sequence $t_i\to 0^-$ so that for each $i$, (a) $\Lambda^R\not\subset\mathcal{N}_\epsilon (M^R_{t_i})$ or (b) $M^R_{t_i}\not\subset\mathcal{N}_\epsilon (\Lambda^R)$. 

If there is a subsequence $t_{i_k}\to 0^-$ for which (a) holds, then there is a sequence, $\xX_{i_k}$, of points in $\Lambda^R\setminus\mathcal{N}_\epsilon (M^R_{t_{i_k}})$. As $\Lambda^R$ is compact, after passing to a subsequence and relabeling, $\xX_{i_k}\to\xX_0$ for some $\xX_0\in\Lambda^R$. By the choice of $\xX_{i_k}$, $\xX_0\notin\mathcal{N}_{\frac{\epsilon}{2}} (M^R_{t_{i_k}})$ for $t_{i_k}$ sufficiently close to $0$. In particular, $\xX_0\neq\oO$. Let $\xX^\prime_0=(1-\epsilon(4R)^{-1})\xX_0$.
Thus, by item \eqref{AsympCone}, $\xX^\prime_0\in\Lambda$. However, by the triangle inequality, $B^{n+1}_{\frac{\epsilon}{8}}(\xX^\prime_0)\cap \sqrt{-t_{i_k}}\, \bar{M}=\emptyset$ for $t_{i_k}$ sufficiently close to $0$. Hence it follows from the definition of Gaussian density that $\Theta_{(\xX^\prime_0,0)}(\mathcal{K})=0$. This contradicts the definition of $\Lambda$.

If there is a subsequence $t_{i_k^\prime}\to 0^-$ for which (b) holds, then there is a sequence of points, $\tilde{\xX}_{i^\prime_k}$, such that $\tilde{\xX}_{i^\prime_k}\in M^R_{t_{i^\prime_k}}\setminus\mathcal{N}_\epsilon (\Lambda^R)$. After passing to a subsequence and relabeling, $\tilde{\xX}_{i^\prime_k}\to\tilde{\xX}_0$ for some $\tilde{\xX}_0\in\bar{B}^{n+1}_R\setminus\set{\oO}$. However, for $t_{i^\prime_k}$ sufficiently close to $0$, $\tilde{\xX}_{i^\prime_k}\in M^R_{t_{i^\prime_k}}\setminus B^{n+1}_{\frac{\epsilon}{2}}$, so $\Theta_{(\tilde{\xX}_{i^\prime_k},t_{i^\prime_k})}(\mathcal{K})=1$. Thus, by the upper semi-continuity of Gaussian density, $\Theta_{(\tilde{\xX}_0,0)}(\mathcal{K})\geq 1$, implying $\tilde{x}_0\in\Lambda^R$. Hence, $\tilde{x}_{i^\prime_k}\in\mathcal{N}_\epsilon(\Lambda^R)$ for $i^\prime_k$ large. This is a contradiction.

Therefore, $\lim_{t\to 0^-} M_t^R=\Lambda^R$ in the Hausdorff metric. As $R$ is arbitrary, item \eqref{HausdorffAsymp} follows immediately.
\end{proof}

To study the regularity of the asymptotic cones of self-shrinking ends, as a first step, we prove a result on the structure of tangent flows to the associated Brakke flows at points in the cones.
\begin{prop} \label{SplitProp}
Suppose that $M$ is an end of a noncompact self-shrinker in $\Real^{n+1}$. Let $\mathcal{K}=\set{\mu_t}_{t\in\Real}$ be the associated Brakke flow for $M$, and let $\Lambda$ be the asymptotic cone of $M$. Then given $\yY\in\Lambda$, any tangent flow $\mathcal{T}=\set{\nu_t}_{t\in\Real}$ to $\mathcal{K}$ at $(\yY,0)$ satisfies that $\nu_t=(\mathcal{H}^1\lfloor\Real_\yY)\times\hat{\nu}_t$ where $\set{\hat{\nu}_t}_{t\in\Real}$ is an integral Brakke flow in $\Real^n\times\Real$ that is backwardly self-similar about $(\oO,0)$.
Moreover, if $n=2$ and $M$ has finite genus, then the support of $\nu_{-1}$ is $\Real_\yY\times\Real$ or $\Real_\yY\times\mathbb{S}^1$.
\end{prop}
\begin{proof}
By our definition, there exists a sequence $\bar{\tau}_i\to +\infty$ so that $\mathcal{K}^{(\yY,0),\bar{\tau}_i}\to\mathcal{T}$. It suffices to prove that $\mathcal{T}$ is translation invariant along the direction of $\yY$. This follows from the self-similarity of $\mathcal{K}$ and White's stratification theorem \cite{WhStrata}. The reader may consult the proof of \cite[Lemma 4.4]{BW} for more details.

Finally we restrict to  $n=2$. With the replacement of the monotonicity formula by its local version derived in \cite[Section 10]{WhStrata}, the arguments in the proof of \cite[Theorem 2]{IlmanenSing} give that the support of $\nu_{-1}$ is smooth. By \cite{AL} the only self-shrinkers in $\Real^2$ are $\Real$ and $\mathbb{S}^1$. Combining these with the discussions in the preceding paragraph, the support of $\nu_{-1}$ must be $\Real_{\yY}\times\Real$ or $\Real_{\yY}\times\mathbb{S}^1$.
\end{proof}

\section{Conditional graphical property for self-shrinkers in unit balls}
This section is devoted to the proof of Theorem \ref{GraphThm}. The key ingredient is a bootstrap machinery (cf. Proposition \ref{BootstrapProp}) in which we show an improvement on the oscillation of the unit normal to any self-shrinker with uniformly bounded total curvature and to which the position vector is almost tangent. In order to establish Proposition \ref{BootstrapProp}, we need to first prove a quantitative splitting result for the self-shrinker under consideration; cf. Lemmas \ref{GraphExtLem} and \ref{RadImproveLem}.

Denote by $d_\Gamma$ the geodesic distance in a surface $\Gamma$. Given $\vV\in\Real^3\setminus\set{\oO}$ and $a,b\in\Real$ with $a<b$, let $\mathring{\vV}=|\vV|^{-1}\vV$ and 
$$
I_\vV(a,b)=\set{\xX\in\Real_\vV \colon a<\xX\cdot\mathring{\vV}<b}.
$$

\begin{lem} \label{GraphExtLem}
There exist universal constants $0<\epsilon_0,\eta_0 \ll 1$ such that given $\epsilon\in (0,\epsilon_0)$ and $\zZ\in\Real^3$ with $\epsilon |\zZ|>4$, if $\Gamma$ is a self-shrinker in $B^3_r(\zZ)$ with $\zZ\in\Gamma$, where $r=(\epsilon^{2}|\zZ|)^{-1}$, and it satisfies that
\begin{align} 
\label{SplitEqn} \sup_{\xX\in\Gamma} |\nN_\Gamma(\xX)\cdot\mathring{\xX}| & <\epsilon_0,\mbox{ and} \\
\label{OscNormalEqn} \sup_{\substack{\xX,\xX^\prime\in\Gamma \\ d_\Gamma(\xX,\xX^\prime)<\epsilon r}} |\nN_\Gamma(\xX)- & \nN_\Gamma(\xX^\prime)|  <\epsilon_0,
\end{align}
then $\Gamma$ contains the graph of a function $u$ on the rectangle
$$
4 \eta_0 r \left\{I_{\zZ}(-1,1)\times I_{\wW}(-\epsilon,\epsilon)\right\}+\zZ
$$
with $u(\zZ)=0$ and $|Du|\leq 10^{-1}$, where $\wW=\zZ\times\nN_\Gamma(\zZ)$.
\end{lem}
\begin{proof}
Let $\zeta>1$ be a universal constant to be determined later in the proof. Define $\tau$ to be the largest number so that there is a function $u$ on 
$$
\Omega=r \left\{I_{\zZ}(-\tau^2,\tau^2)\times I_{\wW}(-\epsilon/2,\epsilon/2)\right\}+\zZ
$$
which satisfies that $\zZ\in {\rm graph}(u)\subset\Gamma$ and $|Du|\leq 10^{-1}$.

We first show that $\tau$ is well-defined. Let $\NN$ be the unit normal vector to $\Real_\zZ\times\Real_\wW$ so that $\nN_\Gamma(\zZ)\cdot\NN\geq 0$. By \eqref{SplitEqn}, 
\begin{equation} \label{CenterSplitEqn}
\nN_\Gamma(\zZ)\cdot\NN=\left(1-|\nN_\Gamma(\zZ)\cdot\mathring{\zZ}|^2\right)^{1/2}>\left(1-\epsilon_0^2\right)^{1/2}.
\end{equation}
This together with \eqref{OscNormalEqn} implies that for $\xX\in\Gamma$ with $d_\Gamma(\xX,\zZ)<\epsilon r$,
$$
\nN_\Gamma(\xX)\cdot\NN=\nN_\Gamma(\zZ)\cdot\NN-\left(\nN_\Gamma(\zZ)-\nN_\Gamma(\xX)\right)\cdot\NN>\left(1-\epsilon_0^2\right)^{1/2}-\epsilon_0.
$$
Thus, if $\epsilon_0<10^{-8}$, then 
$$
\left(1-\epsilon_0^2\right)^{1/2}-\epsilon_0\geq \left(1+10^{-2}\right)^{-1/2},
$$
and $\Gamma$ contains the graph of a function $\bar{u}$ on 
$$
\frac{\epsilon r}{2}\left\{I_{\zZ}(-1,1)\times I_{\wW}(-1,1)\right\}+\zZ
$$
with $\bar{u}(\zZ)=0$ and $|D\bar{u}|\leq 10^{-1}$. The gradient estimate is deduced from that
$$
1+|D\bar{u}|^2=(\nN_{{\rm graph}(\bar{u})}\cdot\NN)^{-2}\leq 1+10^{-2}.
$$
Hence it follows that $\tau$ exists and, moreover, $\tau^2\geq\epsilon/2$.

Next we argue by contradiction to bound $\tau$ from below by a fixed positive number. Through the arguments, assume that 
\begin{equation} \label{AssumpEqn}
\epsilon_0<\min\set{10^{-8}, (4\zeta^2)^{-1}} \mbox{ and } \tau<(4\zeta)^{-1}. 
\end{equation}
We will estimate the oscillation of $\nN_\Gamma$ along the curve $a\mapsto {\rm graph}(u(a\mathring{\zZ}))$. To do this, we need to invoke the fact that self-shrinkers generate mean curvature flows by scaling. By the assumptions on $\epsilon_0$ and $\tau$ we can define
$$
v(\pP,t)=\sqrt{1-t}\, u\left(\frac{\pP}{\sqrt{1-t}}\right)
$$
on the space-time domain
$$
\left\{\zeta \tau \epsilon r\left(I_{\zZ}(-1,1)\times I_{\wW}(-1,1)\right)+\qQ\right\}\times [0,T],
$$
where
$$
\qQ=\left(|\zZ|-\tau^2 r+\zeta \tau \epsilon r\right)\mathring{\zZ} \mbox{ and } T=1-\left(\frac{|\zZ|-\tau^2 r+2\zeta \tau \epsilon r}{|\zZ|+\tau^2 r}\right)^2.
$$
Moreover, by \eqref{AssumpEqn} and $(\epsilon r)^{-1}=\epsilon |\zZ|>4$, we have that
\begin{equation} 
\label{TimeUpperEqn} T=1-\left(1+\frac{2\zeta \tau \epsilon r-2 \tau^2 r}{|\zZ|+\tau^2 r}\right)^2 <\frac{4 \tau^2 r}{|\zZ|+\tau^2 r}<(2 \tau \epsilon r)^2<2^{-6}.
\end{equation}
Since ${\rm graph}(u)\subset\Gamma$ satisfies the self-shrinker equation \eqref{ShrinkerEqn}, the ${\rm graph}(v(\cdot,t))$ evolves by mean curvature.

Let $\zZ^\prime=\qQ+u(\qQ)\NN$ and $\Gamma_t={\rm graph}(v(\cdot,t))$ for $t\in [0,T]$. Next we prove a local gradient estimate for $\Gamma_t$ with respect to the plane $T_{\zZ^\prime}\Gamma$. Denote by $\mathcal{S}$ the space-time track of the flow $\set{\Gamma_t}_{t\in [0,T]}$; that is,
$$
\mathcal{S}=\set{(\xX,t)\in\Real^{n+1}\times [0,T]\colon \xX\in\Gamma_t}.
$$
Define a function $f$ on $\mathcal{S}$ by
$$
f(\xX,t)=\left(\nN_{\Gamma_t}(\xX)\cdot\nN_\Gamma(\zZ^\prime)\right)^{-1}.
$$
We show that $f$ is well-defined. As $|Du(\qQ)|\leq 10^{-1}$, 
\begin{equation} \label{NormalDiffEqn}
|\nN_\Gamma(\zZ^\prime)-\NN|^2=2-\frac{2}{\left(1+|Du(\qQ)|^2\right)^{1/2}}<\frac{1}{4}.
\end{equation}
Observe that
\begin{equation} \label{GradEqn}
|Dv(\pP,t)|=\left\vert Du\left(\frac{\pP}{\sqrt{1-t}}\right)\right\vert\leq 10^{-1},
\end{equation}
where $Dv$ is the spatial gradient of $v$. 
Thus, by \eqref{NormalDiffEqn} and \eqref{GradEqn}, for any $\xX\in\Gamma_t$, 
$$
\nN_{\Gamma_t}(\xX)\cdot\nN_\Gamma(\zZ^\prime)=\nN_{\Gamma_t}(\xX)\cdot\NN+\nN_{\Gamma_t}(\xX)\cdot\left(\nN_\Gamma(\zZ^\prime)-\NN\right)>\frac{10}{\sqrt{101}}-\frac{1}{2}>0,
$$
proving the claim.

Let $\gamma$ be the curve in $\mathcal{S}$ defined by 
$$
\gamma(t)=\left(\qQ+v(\qQ,t)\NN,t\right)\mbox{ for $t\in [0,T]$}.
$$
Next we bound $f|_\gamma$ from above using the gradient estimate of Ecker-Huisken. Namely, it follows from \cite[Theorem 2.1]{EHInterior} that
\begin{equation} \label{EHGradEqn}
\sup_{(\xX,t)\in\gamma} \left((\zeta \tau \epsilon r)^2-|\xX-\qQ|^2-4t\right) f(\xX,t) \leq (\zeta \tau \epsilon r)^2 \sup_{\xX\in\Gamma_0} f(\xX,0).
\end{equation}
To continue, we will need a $L^\infty$ estimate for $v(\qQ,\cdot)$. Assume $\zeta>6$ from now on. Then it follows from \cite[Lemma 3]{CMGrad}, \eqref{TimeUpperEqn}, and \eqref{GradEqn} that
\begin{equation} \label{C0OscEqn}
\sup_{t\in [0,T]} |v(\qQ,t)-v(\qQ,0)|<4 \tau \epsilon r.
\end{equation}
Define $s$ to be
$$
s=1-|\qQ|^2|\zZ|^{-2},
$$
so $\zZ=(1-s)^{-1/2}\qQ$. Invoking that $\epsilon<\epsilon_0<(4\zeta^2)^{-1}$ (cf. \eqref{AssumpEqn}) and $\epsilon |\zZ|>4$, a straightforward calculation gives $0<s<T$. Thus
\begin{equation} \label{ZeroEqn}
v(\qQ,s)=\sqrt{1-s}\, u(\zZ)=0.
\end{equation}
Hence it follows from the triangle inequality, \eqref{C0OscEqn}, and \eqref{ZeroEqn} that
\begin{equation} \label{HeightEqn}
\sup_{t\in [0,T]} |v(\qQ,t)|\leq 2\sup_{t\in [0,T]} |v(\qQ,t)-v(\qQ,0)|<8 \tau \epsilon r.
\end{equation}
Moreover, by \eqref{GradEqn} and $\tau<(4\zeta)^{-1}$ (cf. \eqref{AssumpEqn}), $d_\Gamma(\xX,\zZ^\prime)<\epsilon r$ for $\xX\in\Gamma_0\subset\Gamma$. Thus the hypothesis \eqref{OscNormalEqn} gives that 
\begin{equation} \label{T0GradEqn}
\inf_{\xX\in\Gamma_0}\nN_{\Gamma_0}(\xX)\cdot\nN_\Gamma(\zZ^\prime)>1-2^{-1}\epsilon_0^2.
\end{equation}
Hence, substituting \eqref{TimeUpperEqn}, \eqref{HeightEqn}, and \eqref{T0GradEqn} into \eqref{EHGradEqn} gives that
$$
\sup_{(\xX,t)\in\gamma} f(\xX,t)<\left(1-10^2\zeta^{-2}\right)^{-1}\left(1-2^{-1}\epsilon_0^2\right)^{-1}.
$$
Choosing $\zeta>10^{10}$, as $\epsilon_0<10^{-8}$ (cf. \eqref{AssumpEqn}), 
\begin{equation} \label{GGradEqn}
\sup_{(\xX,t)\in\gamma} f(\xX,t) < \left(1-10^{-4}\right)^{-1}.
\end{equation}

Therefore, 
\begin{align*}
&\inf_{(\xX,t)\in\gamma} \nN_{\Gamma_t}(\xX)\cdot\NN\geq \nN_\Gamma(\zZ)\cdot\nN_\Gamma(\zZ^\prime)-|\nN_\Gamma(\zZ)-\NN|-\sup_{(\xX,t)\in\gamma} |\nN_{\Gamma_t}(\xX)-\nN_\Gamma(\zZ^\prime)| \\
=&f(\sqrt{1-s}\, \zZ,s)^{-1}-\left(2-2\nN_\Gamma(\zZ)\cdot\NN\right)^{1/2}-\sup_{(\xX,t)\in\gamma} \left(2-2f(\xX,t)^{-1}\right)^{1/2} 
>1-10^{-3},
\end{align*}
where we used \eqref{CenterSplitEqn} and $\eqref{GGradEqn}$ with the fact that $(\sqrt{1-s}\, \zZ,s)\in\gamma$ in the last inequality. In particular, this implies that
\begin{equation} \label{EndGradEqn}
\nN_\Gamma(\zZ^\prime)\cdot\NN>1-10^{-3} \mbox{ and } \nN_\Gamma(\zZ^{\prime\prime})\cdot\NN>1-10^{-3},
\end{equation}
where 
$$
\zZ^{\prime\prime}=\frac{\qQ}{\sqrt{1-T}}+u\left(\frac{\qQ}{\sqrt{1-T}}\right)\NN\in {\rm graph}(u).
$$
By the choice of $\zeta$, \eqref{AssumpEqn}, and \eqref{HeightEqn}, both $B^3_{\epsilon r}(\zZ^\prime)$ and $B^3_{\epsilon r}(\zZ^{\prime\prime})$ are contained in $B^3_{r}(\zZ)$. Also, by \eqref{OscNormalEqn} and \eqref{EndGradEqn}, for any $\xX\in\Sigma$ with $d_\Gamma(\xX,\zZ^\prime)<\epsilon r$ or $d_\Gamma(\xX,\zZ^{\prime\prime})<\epsilon r$, 
$$
\nN_\Gamma(\xX)\cdot\NN\geq 1-2\cdot 10^{-3}.
$$ 
Thus $\Gamma$ contains the graphs of functions $u^\prime,u^{\prime\prime}$ defined on, respectively,
\begin{align*}
\Omega^\prime & =\frac{\epsilon r}{2}\left\{I_{\zZ}(-1,1)\times I_{\wW}(-1,1)\right\}+\qQ,  \\
\Omega^{\prime\prime} & =\frac{\epsilon r}{2}\left\{I_{\zZ}(-1,1)\times I_{\wW}(-1,1)\right\}+\frac{\qQ}{\sqrt{1-T}},
\end{align*}
with the properties that $\zZ^\prime\in {\rm graph}(u^\prime), \zZ^{\prime\prime}\in {\rm graph}(u^{\prime\prime})$, and $|Du^\prime|, |Du^{\prime\prime}|\leq 10^{-1}$. 

Finally, observe that by \eqref{AssumpEqn} and the hypothesis that $(\epsilon r)^{-1}=\epsilon |\zZ|>4$,
\begin{align*}
|\qQ|-\frac{\epsilon r}{2} & =|\zZ|-\tau^2 r-\frac{\epsilon r}{2}(1-2\zeta \tau)<|\zZ|-\tau^2 r-\frac{\epsilon r}{8}, \mbox{ and}\\
\frac{|\qQ|}{\sqrt{1-T}}+\frac{\epsilon r}{2} & =|\zZ|+\tau^2 r+\frac{\epsilon r}{2}\left(1-\frac{2\zeta \tau (|\zZ|+\tau^2 r)}{|\zZ|-\tau^2 r+2\zeta \tau \epsilon r}\right)>|\zZ|+\tau^2 r+\frac{\epsilon r}{8}.
\end{align*}
As $\Gamma$ is embedded, $u=u^\prime$ on $\Omega\cap\Omega^\prime$ and $u=u^{\prime\prime}$ on $\Omega\cap\Omega^{\prime\prime}$. Hence $\Gamma$ contains the graph of an extension of $u$ defined on a strictly larger rectangle 
$$
 r \left\{I_{\zZ}(-\tau^2-2^{-3}\epsilon,\tau^2+2^{-3}\epsilon)\times I_{\wW}(-2^{-1}\epsilon,2^{-1}\epsilon)\right\}+\zZ
$$
with its gradient bounded by $10^{-1}$. This contradicts the maximality of $\tau$. 

Therefore the lemma follows with $\epsilon_0=\min\set{10^{-9}, (8\zeta^2)^{-1}}$ and $\eta_0=(8\zeta)^{-2}$, where $\zeta$ is any fixed constant larger than $10^{10}$.
\end{proof}

\begin{lem} \label{RadImproveLem}
Let $u$ be the function given in Lemma \ref{GraphExtLem}. Then
$$
\epsilon^{-1} |\partial_\zZ u|+\epsilon r |\partial_\wW^2u|+r |\partial_\zZ^2u|+r |\partial_\zZ\partial_\wW u|\leq C_0
$$
on the domain
$$
\eta_0 r\left\{I_\zZ(-1,1)\times I_\wW(-\epsilon,\epsilon)\right\}+\zZ
$$
for some universal constant $C_0>0$. 
Here $\partial_\zZ$ and $\partial_\wW$ are to take the directional derivatives along $\mathring{\zZ}$ and $\mathring{\wW}$, respectively.
\end{lem}

\begin{proof}
Without loss of generality we assume that $\zZ=(z,0,0)$ (where $z=|\zZ|$) and $\wW=(0,|\wW|,0)$, so $\Real_\zZ\times\Real_\wW$ is the $x_1x_2$-plane.
As $0<\eta_0 \ll 1$ and $(\epsilon r)^{-1}=\epsilon z>4$, we can define
$$
v\colon (z-4\eta_0 r, z-4\eta_0 r+4\eta_0 \epsilon r)\times(-2\eta_0\epsilon r,2\eta_0\epsilon r)\times [0,T]\to\Real
$$
by
$$
v(x_1,x_2,t)=\sqrt{1-t}\, u\left(\frac{x_1}{\sqrt{1-t}},\frac{x_2}{\sqrt{1-t}}\right),
$$
where 
\begin{equation} \label{TimeUppEqn}
0<T=1-\left(\frac{z-4\eta_0 r+2\eta_0 \epsilon r}{z+\eta_0 r}\right)^2<10\eta_0\epsilon^2r^2<\frac{5}{8}.
\end{equation}
As ${\rm graph}(u)$ satisfies the self-shrinker equation \eqref{ShrinkerEqn}, the ${\rm graph}(v(\cdot,t))$ flows by mean curvature. 

Let
$$
T^\prime=1-\left(\frac{z-4\eta_0 r+2\eta_0 \epsilon r}{z-\eta_0 r}\right)^2.
$$
By our hypotheses 
\begin{equation} \label{TimeLowEqn}
T^\prime=1-\left(1-\frac{(3-2\epsilon)\eta_0 r}{z-\eta_0 r}\right)^2=\frac{(3-2\epsilon)\eta_0 r}{z-\eta_0 r}\left(2-\frac{(3-2\epsilon)\eta_0 r}{z-\eta_0 r}\right)\geq\eta_0 \epsilon^2 r^2.
\end{equation}
By the chain rule of taking derivatives
\begin{equation} \label{SpatialGradEqn}
|Dv(x_1,x_2,t)|=\left\vert Du \left(\frac{x_1}{\sqrt{1-t}},\frac{x_2}{\sqrt{1-t}}\right)\right\vert\leq 10^{-1}
\end{equation}
where $Dv$ is the spatial gradient. Observe that
\begin{equation} \label{VanishEqn}
v(z-4\eta_0 r+2\eta_0\epsilon r,0,t_0)=\sqrt{1-t_0}\, u(z,0,0)=0 
\end{equation}
for some $t_0\in (T^\prime,T)$.
Hence, invoking \eqref{TimeUppEqn}, \eqref{TimeLowEqn}, \eqref{SpatialGradEqn}, and \eqref{VanishEqn}, it follows from the mean curvature flow equation and \cite[Theorem 3.4]{EHInterior} that 
\begin{equation}\label{HighDerEqn}
(\epsilon r)^{-1}|v|+\epsilon r|\partial_t v|+\epsilon r\sum_{i,j=1}^2 |\partial_i\partial_j v|+(\epsilon r)^2 \sum_{i=1}^2 |\partial_t\partial_i v|\leq C
\end{equation}
on the space-time domain
$$
(z-4\eta_0 r+\eta_0 \epsilon r, z-4\eta_0 r+3\eta_0 \epsilon r)\times (-\eta_0\epsilon r,\eta_0\epsilon r)\times [T^\prime,T]
$$
for some constant $C=C(\eta_0)>0$.
Here $\partial_iv$ is the partial derivative of $v$ with respect to $x_i$.

Again, by the chain rule of taking derivatives
\begin{equation} \label{RadGradEqn}
\begin{split}
\partial_t v(x_1,x_2,t) & =\frac{-v(x_1,x_2,t)}{2\left(1-t\right)}+\frac{x_1}{2\left(1-t\right)}\, \partial_1 u\left(\frac{x_1}{\sqrt{1-t}},\frac{x_2}{\sqrt{1-t}}\right)\\
& \quad +\frac{x_2}{2\left(1-t\right)}\, \partial_2 u\left(\frac{x_1}{\sqrt{1-t}},\frac{x_2}{\sqrt{1-t}}\right).
\end{split}
\end{equation}
Furthermore,
\begin{align}
\label{HessEqn} \partial_i\partial_j v(x_1,x_2,t) & =\frac{1}{\sqrt{1-t}}\, \partial_i\partial_j u\left(\frac{x_1}{\sqrt{1-t}},\frac{x_2}{\sqrt{1-t}}\right), \mbox{ and}\\
\begin{split} \label{TSGradEqn}
\partial_t\partial_i v(x_1,x_2,t) & =\frac{x_1}{2\left(1-t\right)^{3/2}}\, \partial_1\partial_i u\left(\frac{x_1}{\sqrt{1-t}},\frac{x_2}{\sqrt{1-t}}\right) \\
&\quad +\frac{x_2}{2\left(1-t\right)^{3/2}}\, \partial_2\partial_i u\left(\frac{x_1}{\sqrt{1-t}},\frac{x_2}{\sqrt{1-t}}\right).
\end{split}
\end{align}
Observe that given $(x_1,x_2)\in (z-\eta_0 r,z+\eta_0 r)\times (-\eta_0\epsilon r,\eta_0\epsilon r)$, there exists a $t\in (T^\prime,T)$ such that 
$$
\left(\sqrt{1-t}\, x_1,\sqrt{1-t}\, x_2\right)\in (z-4\eta_0 r+\eta_0 \epsilon r, z-4\eta_0 r+3\eta_0 \epsilon r)\times (-\eta_0\epsilon r,\eta_0 \epsilon r).
$$
Therefore the lemma follows from substituting \eqref{RadGradEqn}, \eqref{HessEqn}, and \eqref{TSGradEqn} into \eqref{HighDerEqn} and rearranging the inequality.
\end{proof}

Using Lemmas \ref{GraphExtLem} and \ref{RadImproveLem} we now establish the bootstrap machinery mentioned in the beginning of the section.
\begin{prop} \label{BootstrapProp}
There is a universal constant $\epsilon_1\in (0,1)$ such that given $\kappa>0$ and $\beta\in (0,1)$ there is an $\eta_1=\eta_1(\beta,\kappa)\in (0,1)$ such that for any $\tilde{\epsilon}\in (0,\epsilon_1)$ and $\xX_1\in\Real^3$ with $\tilde{\epsilon}^2|\xX_1|>1$, if $\tilde{\Gamma}$ is a self-shrinker in $B^3_{\sigma}(\xX_1)$ with $\xX_1\in\tilde{\Gamma}$, where $\sigma=(\tilde{\epsilon}^{2}|\xX_1|)^{-1}$, and it satisfies that
\begin{align}
\label{BSplitEqn} \sup_{\xX\in\tilde{\Gamma}} |\nN_{\tilde{\Gamma}}(\xX)\cdot\mathring{\xX}| & <\epsilon_1, \\
\label{BOscNormalEqn} \sup_{\substack{\xX,\xX^\prime\in\tilde{\Gamma} \\ d_{\tilde{\Gamma}}(\xX,\xX^\prime)<\tilde{\epsilon}\sigma}} |\nN_{\tilde{\Gamma}}(\xX)-\nN_{\tilde{\Gamma}} & (\xX^\prime)| <\epsilon_1, \mbox{ and} \\
\label{BTotalCurvEqn} \int_{\tilde{\Gamma}} |A_{\tilde{\Gamma}}|^2\, d\mathcal{H}^2 & <\kappa,
\end{align}
then it follows that
$$
\sup_{\substack{\xX\in\tilde{\Gamma} \\ d_{\tilde{\Gamma}}(\xX,\xX_1)<\eta_1\sigma}} |\nN_{\tilde{\Gamma}}(\xX)-\nN_{\tilde{\Gamma}}(\xX_1)|<\beta\epsilon_1.
$$
\end{prop}
\begin{proof}
First we show that if $\epsilon_1$ is sufficiently small, then given $\zZ\in\tilde{\Gamma}\cap B^3_{\frac{\sigma}{2}}(\xX_1)$ with $|\nN_{\tilde{\Gamma}}(\zZ)-\nN_{\tilde{\Gamma}}(\xX_1)|<\beta\epsilon_1$ there is a function $\bar{u}$ on 
$$
\eta\beta\epsilon_1\sigma\left\{I_{\xX_1}(-1,1)\times I_{\wW_1}(-\tilde{\epsilon},\tilde{\epsilon})\right\}+\zZ
$$
satisfying that $\zZ\in {\rm graph}(\bar{u})\subset\tilde{\Gamma}$ and 
$$
\sup_{\xX\in {\rm graph}(\bar{u})} |\nN_{\tilde{\Gamma}}(\xX)-\nN_{\tilde{\Gamma}}(\zZ)|<\frac{\beta\epsilon_1}{8},
$$
where $\eta=\eta(\eta_0,C_0)\in (0,1)$ and $\wW_1=\xX_1\times\nN_{\tilde{\Gamma}}(\xX_1)$.

To see this we first observe that $\frac{1}{2}|\xX_1|<|\zZ|<2|\xX_1|$ as $\tilde{\epsilon}, \sigma<1$.  In order to apply Lemmas \ref{GraphExtLem}, \ref{RadImproveLem} with $\epsilon=2\tilde{\epsilon}$, $r=(4\tilde{\epsilon}^2|\zZ|)^{-1}$, and $\Gamma=\tilde{\Gamma}\cap B^3_r(\zZ)$, we assume that
\begin{equation} \label{EpsilonAssEqn}
\epsilon_1<\min\set{1/4, \epsilon_0/2}.
\end{equation}
Thus there is an $\tilde{\eta}=\tilde{\eta}(\eta_0,C_0)\in (0,1)$ and a function $u$ on 
$$
2\tilde{\eta}\beta\epsilon_1\sigma\left\{I_{\zZ}(-1,1)\times I_\wW(-\tilde{\epsilon},\tilde{\epsilon})\right\}+\zZ,
$$
where $\wW=\zZ\times\nN_{\tilde{\Gamma}}(\zZ)$, such that $\zZ\in {\rm graph}(u)\subset\tilde{\Gamma}$, $|Du|\leq 10^{-1}$,
\begin{equation} \label{HeightOscNorEqn}  
\Vert u\Vert_{L^\infty}<\beta\epsilon_1\tilde{\epsilon}\sigma, \mbox{ and} 
\sup_{\xX\in {\rm graph}(u)} |\nN_{\tilde{\Gamma}}(\xX)-\nN_{\tilde{\Gamma}}(\zZ)|<\frac{\beta\epsilon_1}{8}.
\end{equation}

Next we rewrite ${\rm graph}(u)$ as the graph of a function over some subset of the plane $\Pi_1=\zZ+\Real_{\xX_1}\times\Real_{\wW_1}$. Let $\NN_1$ be the unit normal to $\Pi_1$ so that $\nN_{\tilde{\Gamma}}(\xX_1)\cdot\NN_1\geq 0$, and let $\NN$ be the unit normal to $\Real_{\zZ}\times\Real_{\wW}$ so that $\nN_{\tilde{\Gamma}}(\zZ)\cdot\NN\geq 0$. Observe that by \eqref{BSplitEqn} and the triangle inequality
\begin{equation} \label{BNormalDiffEqn}
\begin{split}
|\NN-\NN_1| & \leq |\NN-\nN_{\tilde{\Gamma}}(\zZ)|+|\nN_{\tilde{\Gamma}}(\zZ)-\nN_{\tilde{\Gamma}}(\xX_1)|+|\NN_1-\nN_{\tilde{\Gamma}}(\xX_1)| \\
& <2\left(2-2\left(1-\epsilon_1^2\right)^{1/2}\right)^{1/2}+\beta\epsilon_1.
\end{split}
\end{equation}
This together with \eqref{BSplitEqn} and \eqref{HeightOscNorEqn} further implies that for all $\xX\in {\rm graph}(u)$,
\begin{equation} \label{BGradEqn}
\begin{split}
\nN_{\tilde{\Gamma}}(\xX)\cdot\NN_1 & \geq\nN_{\tilde{\Gamma}}(\zZ)\cdot\NN-|\NN-\NN_1|-|\nN_{\tilde{\Gamma}}(\xX)-\nN_{\tilde{\Gamma}}(\zZ)| \\
& >\left(1-\epsilon_1^2\right)^{1/2}-2\left(2-2\left(1-\epsilon_1^2\right)^{1/2}\right)^{1/2}-2\beta\epsilon_1.
\end{split}
\end{equation}

Let $\gamma$ be the curve in ${\rm graph}(u)$ given by
$$
\gamma(s)=\zZ+s\mathring{\zZ}+u(\zZ+s\mathring{\zZ})\NN \mbox{ for $|s|<\tilde{\eta}\beta\epsilon_1\sigma$}.
$$
We show that the orthogonal projection of $\gamma$ to $\Pi_1$ stays close to the $\xX_1$-axis. As $\zZ\in B^3_{\frac{\sigma}{2}}(\xX_1)$ and $(\tilde{\epsilon}^2|\xX_1|)^{-1}=\sigma<1$, it follows from the triangle inequality that
$$
|\mathring{\zZ}-\mathring{\xX}_1|<\tilde{\epsilon}^2 \sigma^2<\tilde{\epsilon}^2.
$$
This together with \eqref{HeightOscNorEqn} and \eqref{BNormalDiffEqn} implies that 
\begin{equation} \label{HorizonEqn}
|(\gamma(s)-\zZ)\cdot\mathring{\wW}_1|\leq |s||\mathring{\zZ}-\mathring{\xX}_1|+ \Vert u\Vert_{L^\infty} |\NN-\NN_1|<C\beta\epsilon_1^2\tilde{\epsilon}\sigma
\end{equation}
for some universal constant $C>1$. Similarly, we have that
\begin{equation} \label{RadialEqn}
|(\gamma(s)-\zZ)\cdot\mathring{\xX}_1-s|<C\beta\epsilon_1^2\tilde{\epsilon} \sigma.
\end{equation}

Hence, in addition to \eqref{EpsilonAssEqn}, assuming $\epsilon_1<\tilde{\eta}(10^{4}C)^{-1}$, it follows from \eqref{BGradEqn}, \eqref{HorizonEqn}, and \eqref{RadialEqn} that $\mathcal{N}$, the tubular neighborhood of $\gamma$ in $\tilde{\Gamma}$ with radius $\tilde{\eta}\beta\epsilon_1\tilde{\epsilon}\sigma$,  contains the graph of a function defined on 
$$
\frac{1}{2}\tilde{\eta}\beta\epsilon_1 \sigma\left\{I_{\xX_1}(-1,1)\times I_{\wW_1}(-\tilde{\epsilon},\tilde{\epsilon})\right\}+\zZ.
$$
Invoking the fact that $\mathcal{N}\subset {\rm graph}(u)$, the claim follows with $\eta=\tilde{\eta}/2$.

Next let $\zeta=\eta\beta\epsilon_1\sigma$. Define $\tau$ to be the largest number so that there is a function $\tilde{u}$ defined on 
$$
I_{\xX_1}(-\zeta,\zeta)\times I_{\wW_1}(-\zeta\tau,\zeta\tau)+\xX_1
$$
which satisfies that $\xX_1\in {\rm graph}(\tilde{u})\subset\tilde{\Gamma}$ and 
$$
\sup_{\xX\in {\rm graph}(\tilde{u})} |\nN_{\tilde{\Gamma}}(\xX)-\nN_{\tilde{\Gamma}}(\xX_1)|<\beta\epsilon_1.
$$
It follows from the preceding claim that $\tau$ exists and $\tau\geq\tilde{\epsilon}$. Moreover, by \eqref{BSplitEqn}, 
$$
\nN_{\tilde{\Gamma}}(\xX)\cdot\NN_1=\nN_{\tilde{\Gamma}}(\xX_1)\cdot\NN_1+\left(\nN_{\tilde{\Gamma}}(\xX)-\nN_{\tilde{\Gamma}}(\xX_1)\right)\cdot\NN_1
>\left(1-\epsilon_1^2\right)^{1/2}-\beta\epsilon_1
$$
for all $\xX\in {\rm graph}(\tilde{u})$. This together with our hypothesis on $\epsilon_1$ implies that
\begin{equation} \label{BC1Eqn}
|D\tilde{u}|^2=\left(\nN_{{\rm graph}(\tilde{u})}\cdot\NN_1\right)^{-2}-1<1.
\end{equation}

Finally we bound $\tau$ from below in terms of $\epsilon_1$, $\beta$, and $\kappa$. First let us parametrize ${\rm graph}(\tilde{u})$ as follows. Define
$$
\Psi\colon (-\zeta,\zeta)\times (-\zeta\tau,\zeta\tau)\to\Real^3
$$
by
$$
\Psi(a,b)=\xX_1+a\mathring{\xX}_1+b\mathring{\wW}_1+\tilde{u}(\xX_1+a\mathring{\xX}_1+b\mathring{\wW}_1)\NN_1.
$$
If $\tau<1/4$, then it follows from \eqref{BC1Eqn} that $|\tilde{u}(\Psi(0,\pm\zeta\tilde{\tau}))|<\sigma/4$, where $\tilde{\tau}=\tau-2^{-1}\tilde{\epsilon}$, and thus $\Psi(0,\pm\zeta\tilde{\tau})\in\tilde{\Gamma}\cap B^3_{\frac{\sigma}{2}}(\xX_1)$. Invoking the preceding claim (with $\zZ=\Psi(0,\pm\zeta\tilde{\tau})$) and the maximality of $\tau$ we get that 
\begin{align*}
\inf_{a\in (-\zeta,\zeta)} \left\vert\nN_{\tilde{\Gamma}}\left(\Psi(a,\zeta\tilde{\tau})\right)-\nN_{\tilde{\Gamma}}\left(\Psi(a,0)\right)\right\vert & >\frac{\beta\epsilon_1}{4},\mbox{ or} \\
\inf_{a\in (-\zeta,\zeta)} \left\vert\nN_{\tilde{\Gamma}}\left(\Psi(a,-\zeta\tilde{\tau})\right)-\nN_{\tilde{\Gamma}}\left(\Psi(a,0)\right)\right\vert & >\frac{\beta\epsilon_1}{4}.
\end{align*}
Without loss of generality we assume that the former inequality holds. Hence it follows from the Cauchy-Schwarz inequality, \eqref{BTotalCurvEqn}, and \eqref{BC1Eqn} that
\begin{align*}
\frac{1}{2}\zeta\beta\epsilon_1 & <\int_{-\zeta}^\zeta\left\vert\nN_{\tilde{\Gamma}}\left(\Psi(a,\zeta\tilde{\tau})\right)-\nN_{\tilde{\Gamma}}\left(\Psi(a,0)\right)\right\vert da \leq\sqrt{2} \int_{{\rm graph}(\tilde{u})} |A_{\tilde{\Gamma}}|\, d\mathcal{H}^2 \\
& \leq\sqrt{2}\left(\int_{{\rm graph}(\tilde{u})} |A_{\tilde{\Gamma}}|^2 \, d\mathcal{H}^2\right)^{1/2} \left(\int_{{\rm graph}(\tilde{u})} 1 \, d\mathcal{H}^2\right)^{1/2} \leq 4\zeta (\tau\kappa)^{1/2},
\end{align*}
giving that $\tau\geq (\beta\epsilon_1)^2(2^{6}\kappa)^{-1}$.

Therefore the proposition follows with 
$$
\epsilon_1=\min\set{2^{-3}\epsilon_0, \eta(10^{5}C)^{-1}} \mbox{ and } \eta_1=\min\set{2^{-2}\eta\beta\epsilon_1, \eta(\beta\epsilon_1)^3(2^{6}\kappa)^{-1}},
$$
where $C>1$ is some universal constant and $\eta=\eta(\eta_0,C_0)\in (0,1)$. 
\end{proof}

Finally we repeatedly apply Proposition \ref{BootstrapProp} to conclude Theorem \ref{GraphThm}.
\begin{proof}[Proof of Theorem \ref{GraphThm}]
First we choose $\beta\in (0,1)$ in Proposition \ref{BootstrapProp} as follows:
$$
\beta=\min\set{2^{-1},\epsilon_1^{-1} \left(2-2\left(1+\delta^2\right)^{-1/2}\right)^{1/2}}.
$$
Let  $\eta_1=\eta_1(\beta,\kappa)\in (0,1)$ be the constant given in Proposition \ref{BootstrapProp}, and define
$$
\tilde{\epsilon}_0=\min\set{2^{-4}\eta_1, 2^{-2}\epsilon_1}.
$$
We choose 
$$
\alpha=2^{-1}\tilde{\epsilon}_0\epsilon_1\mbox{ and } \mathcal{R}=2(\tilde{\epsilon}_0)^{-2},
$$
so, by \eqref{SlcurvEqn}, the self-shrinker $\Sigma$ in the theorem satisfies that
\begin{align}
\label{GSplitEqn} &  \sup_{\xX\in\Sigma} |\nN_\Sigma(\xX)\cdot\mathring{\xX}|\leq 2\sqrt{2} \sup_{\xX\in\Sigma}|\xX|^{-1} |A_\Sigma(\xX)|<\epsilon_1, \mbox{ and } \\
\label{GOscNormalEqn} & \sup_{\substack{\xX,\xX^\prime\in\Sigma \\ d_\Sigma(\xX,\xX^\prime)<\tilde{\epsilon}_0r_0}} |\nN_\Sigma(\xX)-\nN_\Sigma(\xX^\prime)|< \tilde{\epsilon}_0 r_0 \sup_{\xX\in\Sigma} |A_\Sigma(\xX)|<\epsilon_1,
\end{align}
where $r_0=(\tilde{\epsilon}^{2}_0 |\xX_0|)^{-1}<1/2$ and we used that $|\xX|<2|\xX_0|$ if $\xX\in B^3_1(\xX_0)$.

Let $\tilde{\epsilon}_i=2^{-i}\tilde{\epsilon}_0$ and $r_i=4^{i}r_0$ for $i\in\mathbb{N}$. Define $\rho_0=1$ and, inductively,
$$
\rho_{i+1}=\rho_{i}-2^{-1}r_i.
$$
We show that for $i\in\mathbb{N}$, if $r_i<1/2$, then
$$
\sup_{\substack{\xX,\xX^\prime\in\Sigma\cap B^3_{\rho_i}(\xX_0) \\ d_\Sigma(\xX,\xX^\prime)<\tilde{\epsilon}_ir_i}} |\nN_\Sigma(\xX)-\nN_\Sigma(\xX^\prime)|<\epsilon_1.
$$
For $i=0$ this is true in view of \eqref{GOscNormalEqn}. Inductively, assume that the claim holds for $i$. Also, we assume $r_{i+1}<1/2$; otherwise, it is vacuously true. As $|\xX_0|>\mathcal{R}>2$, observe that $2^{-1}|\xX_0|<|\xX|<2|\xX_0|$ if $\xX\in B_1^3(\xX_0)$. This together with our choice of $\tilde{\epsilon}_0$ and the triangle inequality implies that for any $\xX\in B^3_{\rho_{i+1}}(\xX_0)$, 
$$
2\tilde{\epsilon}_i<\epsilon_1, \; (2\tilde{\epsilon}_i |\xX|)^{-1}<\tilde{\epsilon}_i r_i, \mbox{ and } B^3_{\frac{1}{4\tilde{\epsilon}_i^2|\xX|}}(\xX)\subset B^3_{\rho_i}(\xX_0).
$$
Thus, invoking \eqref{TotalcurvEqn}, \eqref{GSplitEqn}, and the inductive assumption, we apply Proposition \ref{BootstrapProp} with $\xX_1=\xX$, $\tilde{\epsilon}=2\tilde{\epsilon}_i$, $\sigma=(4\tilde{\epsilon}^2_i |\xX|)^{-1}$, and $\tilde{\Gamma}=\Sigma\cap B^3_{\sigma}(\xX)$, where $\xX$ is any point in $\Sigma\cap B^3_{\rho_{i+1}}(\xX_0)$. It follows that 
\begin{equation} \label{IOscNormalEqn}
\sup_{\substack{\xX^\prime\in\Sigma\cap B^3_\sigma(\xX) \\ d_\Sigma (\xX^\prime,\xX)<\eta_1\sigma}} |\nN_\Sigma(\xX)-\nN_\Sigma(\xX^\prime)|<\beta\epsilon_1.
\end{equation}
Note that by our choices of $\tilde{\epsilon}_0$ and $\mathcal{R}$,
$$
\eta_1\sigma>\eta_1 (8\tilde{\epsilon}^2_i|\xX_0|)^{-1}\geq \tilde{\epsilon}_{i+1}r_{i+1},
$$
Hence the claim for $i+1$ follows immediately from \eqref{IOscNormalEqn} and the arbitrariness of the point $\xX$ in $\Sigma\cap B^3_{\rho_{i+1}}(\xX_0)$.

Finally let 
$$
i_0=\max\set{i\in\mathbb{N}\colon r_i<1/2}.
$$
Clearly $i_0$ exists, and by the maximality of $i_0$ we have that $2^{-3}\leq r_{i_0}<2^{-1}$. Observe that by the definitions of $r_i$ and $\rho_i$,
$$
\rho_{i_0}=1-\frac{r_{i_0}}{2}\sum_{i=0}^{i_0} 4^{-i}>1-r_{i_0}>\frac{1}{2}.
$$
Again, in view of \eqref{TotalcurvEqn}, \eqref{GSplitEqn}, and the preceding claim for $i_0$, we apply Proposition \ref{BootstrapProp} with $\xX_1=\xX_0$, $\tilde{\epsilon}=\tilde{\epsilon}_{i_0}$, $\sigma=r_{i_0}$, and $\tilde{\Gamma}=\Sigma\cap B^3_{r_{i_0}}(\xX_0)$ to conclude that
$$
\sup_{\substack{\xX\in\Sigma \\ d_\Sigma(\xX,\xX_0)<\eta_1 r_{i_0}}} |\nN_\Sigma(\xX)-\nN_\Sigma(\xX_0)|<\beta\epsilon_1.
$$
Hence $\Sigma$ contains the graph of a function $\bar{u}$ on  $B^2_{\frac{\eta_1}{16}}(\xX_0)\subset T_{\xX_0}\Sigma$ with $\bar{u}(\xX_0)=0$ and $|D\bar{u}|\leq \delta$. Here the gradient estimate follows from our choice of $\beta $ and the elementary fact that
$$
1+|D\bar{u}|^2=\left(\nN_{{\rm graph}(\bar{u})}\cdot\nN_\Sigma(\xX_0)\right)^{-2}.
$$

Therefore the theorem follows with 
$$
\alpha=\min\set{2^{-5}\eta_1\epsilon_1,2^{-3}\epsilon_1^2}, \; \mathcal{R}=\min\set{2^{9}\eta_1^{-2}, 2^5\epsilon_1^{-2}}, \mbox{ and } \rho=2^{-4}\eta_1,
$$
where $\eta_1=\eta_1(\beta,\kappa)$ is given in Proposition \ref{BootstrapProp} with $\beta$ defined in terms of $\delta$ at the beginning of the proof.
\end{proof}

\section{Sheeting theorem}
The goal of this section is to prove the sheeting theorem, i.e., Theorem \ref{SheetThm}. First we apply Theorem \ref{GraphThm} to prove a sequential version of Theorem \ref{SheetThm}.
\begin{prop} \label{SeqSheetProp}
Let $M$ be an end of a noncompact self-shrinker in $\Real^3$ of finite genus, and let $\mathcal{K}$ be the associated Brakke flow for $M$. Given $\yY\in\Real^3\setminus\set{\oO}$ with $\Theta_{(\yY,0)}(\mathcal{K})\geq 1$, there exist
\begin{itemize}
\item a positive integer $L$, a sufficiently large integer $i_0$,
\item two increasing unbounded sequences $R_i$ and $\tau_i$ with $\tau_{i+1}\tau_i^{-1}\to 1$, and,
\item a sequence $\Sigma^{\tau_i}$ of self-shrinkers in $\Real^3$ with the property that either the $\Sigma^{\tau_i}$ are all $\Real_\yY\times\mathbb{S}^1$, or they are all of the form $\Real_\yY\times\Real$,
\end{itemize} 
depending only on $M$ and $\yY$, such that for each $i>i_0$,
\begin{enumerate}
\item \label{SeqComponent} $(M-\tau_i\yY)\cap B^3_{2R_i}$ decomposes into $L$ connected components, $M_1^{\tau_i},\ldots,M_L^{\tau_i}$;
\item \label{SeqSheet} each $M_j^{\tau_i}$ can be written as the normal exponential graph of a function $f_j^{\tau_i}$ over some subset $\Omega_j^{\tau_i}$ of $\Sigma^{\tau_i}$.
\end{enumerate}
Furthermore, for each $j$,
\begin{equation} \label{SeqC1Eqn}
\lim_{i\to +\infty} \sup_{\Omega_j^{\tau_i}} |f_j^{\tau_i}|+|\nabla_{\Sigma^{\tau_i}} f_j^{\tau_i}|=0.
\end{equation}
\end{prop}
\begin{proof}
Recall the local monotonicity formula from \cite[Section 10]{WhStrata}:
\begin{equation} \label{MonotonEqn}
\begin{split}
& -\frac{d}{dt}(-4\pi t)^{-1} \int_{\sqrt{-t}M}\phi_\yY^2\, \mathrm{e}^{\frac{|\zZ-\yY|^2}{4t}} d\mathcal{H}^2+C \\
\geq & (-4\pi t)^{-1}\int_{\sqrt{-t}M}\left\vert\mathbf{H}_{\sqrt{-t}M}-\frac{(\zZ-\yY)^\perp}{2t}-\frac{2(D\phi_\yY)^\perp}{\phi_\yY}\right\vert^2\phi^2_\yY\, \mathrm{e}^{\frac{|\zZ-\yY|^2}{4t}} d\mathcal{H}^2,
\end{split}
\end{equation}
where $\phi_\yY$ is some cutoff function about $\yY$ and $C$ is a positive constant depending only on $\phi_{\yY}$ and the area ratios bound of $M$. Thus, changing variables in \eqref{MonotonEqn} by setting $s=-\ln(-t)$ and $\xX=\mathrm{e}^{s/2}(\zZ-\yY)$, it follows from integrating the resulted inequality against $s$ and the triangle inequality that 
$$
\int_{s_0}^{+\infty} \int_{M_s} \left\vert\mathbf{H}_{M_s}+\frac{\xX^\perp}{2}\right\vert^2 \phi_{(\yY,s)}^2\, \mathrm{e}^{-\frac{|\xX|^2}{4}} \, d\mathcal{H}^2ds<+\infty
$$
for some $s_0$ sufficiently large, where $M_s=M-\mathrm{e}^{s/2}\yY$ and $\phi_{(\yY,s)}(\xX)=\phi_{\yY}(\yY+\mathrm{e}^{-s/2}\xX)$. 
Hence there is an increasing unbounded sequence $s_i$ so that for each $i$,
$$
\int_{s_i}^{+\infty} \int_{M_s} \left\vert\mathbf{H}_{M_s}+\frac{\xX^\perp}{2}\right\vert^2 \phi_{(\yY,s)}^2\, \mathrm{e}^{-\frac{|\xX|^2}{4}} \, d\mathcal{H}^2ds <i^{-2}.
$$
By the mean value theorem, for each $i$, there is a finite set of numbers,
$$
s_{i,m}\in \left[s_i+i^{-1}(m-1),s_{i}+i^{-1}m\right], 
$$
where $m\in\mathbb{Z}^+$ and $m<i(s_{i+1}-s_i)$, so that
$$
\int_{M_{s_{i,m}}} \left\vert\mathbf{H}_{M_{s_{i,m}}}+\frac{\xX^\perp}{2}\right\vert^2 \phi_{(\yY,s_{i,m})}^2\, \mathrm{e}^{-\frac{|\xX|^2}{4}} \, d\mathcal{H}^2<i^{-1}.
$$
Therefore, after relabeling $s_{i,m}$, we get an increasing sequence $\bar{s}_i\to +\infty$ so that $|\bar{s}_{i+1}-\bar{s}_i|\to 0$ and 
\begin{equation} \label{AsympShrinkerEqn}
\int_{M_{\bar{s}_{i}}} \left\vert\mathbf{H}_{M_{\bar{s}_i}}+\frac{\xX^\perp}{2}\right\vert^2 \phi_{(\yY,\bar{s}_i)}^2\, \mathrm{e}^{-\frac{|\xX|^2}{4}} \, d\mathcal{H}^2\to 0.
\end{equation}
As $M$ satisfies the self-shrinker equation \eqref{ShrinkerEqn}, it follows that $M_{\bar{s}_i}$ satisfies
$$
\mathbf{H}_{M_{\bar{s}_i}}+\frac{1}{2}\left(\xX+\mathrm{e}^{\bar{s}_i/2}\yY\right)^\perp=\oO.
$$
Thus \eqref{AsympShrinkerEqn} gives that
\begin{equation} \label{AsympSplitEqn}
\mathrm{e}^{\bar{s}_i} \int_{M_{\bar{s}_i}} |\yY^\perp|^2 \phi_{(\yY,\bar{s}_i)}^2 \, \mathrm{e}^{-\frac{|\xX|^2}{4}} \, d\mathcal{H}^2 \to 0.
\end{equation}

Set $\tau_i=\mathrm{e}^{\bar{s}_i/2}$. Thus $\tau_i$ is an increasing unbounded sequence with $\tau_i^{-1}\tau_{i+1}\to 1$. Next we show that for all $l\in\mathbb{Z}^+$, 
\begin{equation} \label{SublinearCurvEqn}
\sup_{\xX\in M\cap B^3_l(\tau_i \yY)} |\xX|^{-1} |A_M(\xX)|\to 0, \mbox{ as $i\to +\infty$}.
\end{equation}
To see this we argue by contradiction. Suppose that there was a $l\in\mathbb{Z}^+$, a $\zeta>0$, a subsequence $i_k\to +\infty$, and a sequence of points, $\xX_{i_k}\in M\cap B^3_l(\tau_{i_k}\yY)$, so that
\begin{equation} \label{LowCurvEqn}
|\xX_{i_k}|^{-1} |A_M(\xX_{i_k})|>\zeta>0.
\end{equation}
Let $N_{i_k}$ be the scaling of $M$ about $\xX_{i_k}$ by $|\xX_{i_k}|$; that is,
$$
N_{i_k}=|\xX_{i_k}| (M-\xX_{i_k}).
$$
Thus \eqref{LowCurvEqn} gives
\begin{equation} \label{CenterCurvEqn}
|A_{N_{i_k}}(\oO)|>\zeta>0.
\end{equation}
However, it follows from \cite[Theorem 19]{Song} (cf. Theorem \ref{LinearCurvThm}) that
$$
\sup_{\xX\in M} \left(1+|\xX|\right)^{-1} |A_M(\xX)|<C^\prime
$$
for some constant $C^\prime=C^\prime(M)>0$. Thus, for $i_k$ sufficiently large,
\begin{equation} \label{PointCurvEqn}
\sup_{\xX\in N_{i_k}\cap B^3_{1}} |A_{N_{i_k}}(\xX)|<2C^\prime.
\end{equation}
As $|\xX_{i_k}|\leq \tau_{i_k}|\yY|+l$, it follows from \eqref{AsympSplitEqn} that
\begin{equation} \label{L2SplitEqn}
\int_{N_{i_k}\cap B_1^3} |\yY^\perp|^2 \, d\mathcal{H}^2\to 0.
\end{equation}
As $M$ is an end of a self-shrinker in $\Real^3$, the rescaled surface $N_{i_k}$ satisfies 
\begin{equation} \label{AsympTranslatorEqn}
\mathbf{H}_{N_{i_k}}+\frac{1}{2}\left(\frac{\xX_{i_k}^\perp}{|\xX_{i_k}|}+\frac{\xX^\perp}{|\xX_{i_k}|^2}\right)=\oO.
\end{equation}
Hence, passing to a subsequence and relabeling, it follows from \eqref{PointCurvEqn}, the elliptic regularity theory, and the Arzel\`{a}-Ascoli theorem that $N_{i_k}\cap B_1^3$ converges locally smoothly (possibly with multiplicity) to a surface $N$ in $B_1^3$. Furthermore, by \eqref{L2SplitEqn} and \eqref{AsympTranslatorEqn}, $N$ satisfies that
$$
\yY^\perp=\oO \mbox{ and } \mathbf{H}_{N}+\frac{\yY^\perp}{2|\yY|}=\oO,
$$
which further implies that $N$ is flat. However, \eqref{CenterCurvEqn} gives that $|A_{N}(\oO)|\geq\zeta>0$. This is a contradiction.

As $M$ is an end of a self-shrinker in $\Real^3$, which, in particular, is properly embedded by our definition,  it follows from \cite[Theorem 1.3]{CZ} (cf. \cite[Theorem 1.1]{DX}) that $M$ has bounded area ratios. Thus the local Gauss-Bonnet estimate of Ilmanen \cite[Theorem 3]{IlmanenSing} together with \eqref{AsympShrinkerEqn} implies that given $l\in\mathbb{Z}^+$ there is an $i^\prime_l\in\mathbb{Z}^+$ and a $C^{\prime\prime}_l>0$ such that for all $i>i^\prime_l$,
$$
\int_{M\cap B_l^3(\tau_i\yY)} |A_M|^2 \, d\mathcal{H}^2<C^{\prime\prime}_l.
$$
Hence, combining this with \eqref{SublinearCurvEqn}, it follows from Theorem \ref{GraphThm} that given positive integers $k$ and $l$ there is an $i_{k,l}\in\mathbb{Z}^+$ and a $\rho_{k,l}\in (0,1)$ such that for all $i>i_{k,l}$ and $\xX\in M\cap B^3_{l-1}(\tau_i\yY)$, the connected component of $M\cap B^3_{\rho_{k,l}}(\xX)$ containing $\xX$ can be written as the graph of a function over some subset of $T_\xX M$ with its gradient bounded by $1/k$. 

Finally, in view of \eqref{ScalingShrinkerEqn}, it follows from item \eqref{TangentFlow} of Proposition \ref{WhiteProp}, Proposition \ref{SplitProp}, and \cite[Theorem 1.1]{BWHM} that there is a positive integer $L$ and a sequence $\Sigma^{\tau_i}$ of self-shrinkers in $\Real^3$ with the property that either the $\Sigma^{\tau_i}$ are all $\Real_\yY\times\mathbb{S}^1$, or they are all of the form $\Real_\yY\times\Real$ so that the varifold distance between $M-\tau_i\yY$ and $L\Sigma^{\tau_i}$ converges to $0$. This together with the discussion in the preceding paragraph and the Arzel\`{a}-Ascoli theorem implies that given $l\in\mathbb{Z}^+$ there is an $i^{\prime\prime}_l\in\mathbb{Z}^+$ such that for all $i>i^{\prime\prime}_l$, $(M-\tau_i\yY)\cap B_{l-1}^3$ has $L$ connected components and each  component is given by the exponential normal graph of a function over some subset of $\Sigma^{\tau_i}$ with its $C^1$-norm bounded by $1/l$. Without loss of generality let us assume that $i^{\prime\prime}_l$ is increasing in $l$. Therefore, choosing $i_0=i_1^{\prime\prime}$, and $R_i=1/8$ if $i\leq i_1^{\prime\prime}$ and $R_i=l/4$ if $i_l^{\prime\prime}<i\leq i^{\prime\prime}_{l+1}$, items \eqref{SeqComponent}, \eqref{SeqSheet} of the proposition and \eqref{SeqC1Eqn} follow immediately.
\end{proof}

We will also need a special case of a maximum principle for self-shrinkers due to Song \cite[Proposition 6]{Song}. For the sake of completeness we include it here.
\begin{lem} \label{MaxPrincipleLem}
Given $\vV\in\Real^3\setminus\set{\oO}$, let $\tilde{M}$ be a connected, smooth, properly embedded, codimension-one submanifold of $S=\Real_\vV\times\bar{B}^2_2$ possibly with boundary $\partial\tilde{M}\subset\partial S$. If $\tilde{M}\setminus\partial\tilde{M}$ satisfies the self-shrinker equation \eqref{ShrinkerEqn}, and $\tilde{M}$ and $\partial S$ intersect transversally, then $\tilde{M}\cap\bar{B}^3_2\neq\emptyset$.
\end{lem}
\begin{proof}
We argue by contradiction. Suppose that $\tilde{M}\cap\bar{B}^3_2=\emptyset$. As $\tilde{M}$ is connected, we may also assume that $\tilde{M}\subset\Real_\vV^{+}\times\bar{B}^2_2$. Let 
$$
N=\partial B^3_2\cap\set{\xX\in\Real^3\colon \xX\cdot\vV\geq 0},
$$
and let $P\colon \Real^3\to\Real^3$ be the orthogonal projection onto the static plane $\Pi$ normal to $\vV$.
Define a function $h\colon N\to\Real$ by 
$$
h(\xX)=\min\set{(\zZ-\xX)\cdot\vV\colon\zZ\in\tilde{M} \mbox{ with } P(\zZ)=P(\xX)}.
$$
Here we adopt the convention that $h(\xX)=+\infty$ if there do not exist $\zZ\in\tilde{M}$ so that $P(\zZ)=P(\xX)$. As $\tilde{M}$ is properly embedded, $h$ is well-defined. Moreover, by our assumption, $h>0$. Thus there is a constant $C\geq 0$ such that
$$
\inf\set{h(\xX)\colon\xX\in N}=C.
$$

Take a sequence of points, $\xX_i\in N$, so that $h(\xX_i)\to C$. For each $i$, there is a unique $\zZ_i\in\tilde{M}$ so that $P(\zZ_i)=P(\xX_i)$ and $h(\xX_i)=(\zZ_i-\xX_i)\cdot\vV$. As $\tilde{M}$ is properly embedded, after passing to a subsequence and relabeling, it follows that $\zZ_i\to\zZ_0$ for some $\zZ_0\in\tilde{M}$ and $h(\xX_0)=C$ where $\xX_0\in N$ with $P(\xX_0)=P(\zZ_0)$. We show $\zZ_0\notin\partial\tilde{M}$, implying that $h$ attains its minimum at $\xX_0\in\mathring{N}$. If $\zZ_0\in\partial\tilde{M}$, then $\xX_0\in\partial N$ and thus $\vV$ is parallel to $T_{\xX_0}N$. Observe that $T_{\xX_0}N=T_{\zZ_0}(\partial S)$. As $\tilde{M}$ and $\partial S$ intersect transversally, the claim follows from the transversality of $T_{\xX_0}N$ and $T_{\zZ_0}\tilde{M}$. 

Next we observe that $T_{\xX_0}N=T_{\zZ_0}\tilde{M}$. As $\xX_0\in\mathring{N}$, there exist neighborhoods $V\subset\tilde{M}$ of $\zZ_0$ and $W\subset N$ of $\xX_0$ such that $V$ and $W$ are given by the graphs of functions $f$ and $g$, respectively, over some subset of $\Pi$. Moreover, $f-g$ has a positive local minimum at $P(\xX_0)$ in the interior of its domain. However, as both $V$ and $W$ satisfy the self-shrinker equation \eqref{ShrinkerEqn}, it follows from a direct computation that $f-g$ satisfies 
$$
\sum_{k,l=1}^2 a_{kl} \partial_{k}\partial_{l} (f-g)+\sum_{k=1}^2 b_k \partial_{k} (f-g)+\frac{1}{2} (f-g)=0,
$$
where, without loss of generality, we assume that $\Pi$ is the $x_1x_2$-plane, $\partial_k$ denotes to take the partial derivative with respect to $x_k$, $a_{kl}$ and $b_k$ smoothly depends on $f$, $g$, and their derivatives up to second order, and the matrix $(a_{kl})$ is symmetric and uniformly positive definite.
Hence, by the maximum principle for uniformly elliptic equations, $f-g$ cannot have any positive local minimum in the interior of its domain, leading to a contradiction.
\end{proof}

\begin{proof}[Proof of Theorem \ref{SheetThm}]
By Proposition \ref{SeqSheetProp} and Brakke's regularity theorem \cite[Section 6]{B} (cf. \cite{WhReg}), it suffices to show that, for all $i$ sufficiently large, the connected components of $M\cap B^3_{R_i}(\tau_i\yY)$ extend to infinity and remain disjoint in the solid half-cylinder $I_{\yY}(\tau_i|\yY|,+\infty)\times B^2_{R_i}$.

For $i,k\in\mathbb{Z}^+$ and $j\in\set{1,\ldots,L}$, let
$$
\gamma_j^{i,k}=\left(M_j^{\tau_k}+\tau_k\yY\right)\cap B^3_{R_i}(\tau_k\yY) \cap\set{\xX\in\Real^3\colon\xX\cdot\yY=\tau_k|\yY|^2}.
$$
First we show that given $i$ sufficiently large, for any $k\geq i$ and any connected component  
$$
M^{i,k}\subset M\cap \left(I_{\yY}(\tau_k|\yY|,\tau_{k+1}|\yY|)\times B^2_{R_i}\right),
$$
the two sets
$$
\partial M^{i,k}\cap\set{\xX\in\Real^3\colon\xX\cdot\yY=\tau_k|\yY|^2}
\mbox{ and }
\partial M^{i,k}\cap\set{\xX\in\Real^3\colon\xX\cdot\yY=\tau_{k+1}|\yY|^2}
$$
are unions of the same number of elements of $\set{\gamma_j^{i,k}}$ and $\set{\gamma_j^{i,k+1}}$, respectively. 

Fix any $k\geq i$. We may assume $R_i<(\tau_{k+1}-\tau_k)|\yY|$; otherwise, the claim follows from Proposition \ref{SeqSheetProp}.
As $M$ satisfies the self-shrinker equation \eqref{ShrinkerEqn}, so does $M^{i,k}$. Thus the family, $M_s^{i,k}=M^{i,k}-\mathrm{e}^{s/2}\yY$, satisfies
\begin{equation} \label{RescaleMCFEqn}
\left(\partial_s\xX\right)^\perp=\mathbf{H}_{M_s^{i,k}}+\frac{\xX^\perp}{2}.
\end{equation}
Define 
$$
s^+_k=2\ln\left(\tau_k+(2|\yY|)^{-1}R_i\right) \mbox{ and } s^{-}_{k+1}=2\ln\left(\tau_{k+1}-(2|\yY|)^{-1}R_i\right).
$$
Observe that $s_k^+<s_{k+1}^{-}$ and $B^3_{R_i/2}\cap \partial M^{i,k}_s=\emptyset$ for all $s\in [s_k^+,s^{-}_{k+1}]$.
Let $\psi$ be a cutoff function about $\oO$ which satisfies that $\psi\equiv 1$ in $B^3_{R_i/4}$, $\psi\equiv 0$ outside of $B^3_{R_i/2}$, and $|D\psi|\leq 8R_i^{-1}$. Hence it follows from \eqref{RescaleMCFEqn} and the first variation formula for the area functional that
\begin{align*}
\frac{d}{ds}\int_{M^{i,k}_s} \psi^2 \, \mathrm{e}^{-\frac{|\xX|^2}{4}} \,  d\mathcal{H}^2 = & -\int_{M^{i,k}_s} \left\vert\mathbf{H}_{M^{i,k}_s}+\frac{\xX^\perp}{2}\right\vert^2 \psi^2 \, \mathrm{e}^{-\frac{|\xX|^2}{4}} \, d\mathcal{H}^2 \\
 & +2\int_{M^{i,k}_s} D\psi\cdot\left(\mathbf{H}_{M^{i,k}_s}+\frac{\xX^\perp}{2}\right)\psi \, \mathrm{e}^{-\frac{|\xX|^2}{4}} \, d\mathcal{H}^2.
\end{align*}
Furthermore, integrating this from $s^+_k$ to $s^{-}_{k+1}$ and using the Cauchy-Schwarz inequality give that
\begin{equation} \label{MassDropEqn}
\begin{split}
& \left\vert\int_{M^{i,k}_{s^{+}_{k}}} \psi^2\, \mathrm{e}^{-\frac{|\xX|^2}{4}} \, d\mathcal{H}^2-\int_{M^{i,k}_{s^{-}_{k+1}}} \psi^2\, \mathrm{e}^{-\frac{|\xX|^2}{4}} \, d\mathcal{H}^2\right\vert \\
\leq & 2\int_{s^{+}_k}^{s^{-}_{k+1}} \int_{M^{i,k}_s} \left\vert\mathbf{H}_{M^{i,k}_s}+\frac{\xX^\perp}{2}\right\vert^2 \psi^2 \, \mathrm{e}^{-\frac{|\xX|^2}{4}} \, d\mathcal{H}^2 ds+C\left(s^{-}_{k+1}-s^+_k\right)\mathrm{e}^{-\frac{R_i^2}{64}},
\end{split}
\end{equation}
where $C$ depends only on the area ratios bound of $M$.

Recall, that $i$ is assumed to be sufficiently large and $k\ge i$. By the definitions of $s_k^+$ and $s_{k+1}^{-}$ and by our choice of the sequence $\tau_i$,
\begin{equation} \label{TimeEqn}
0<s_{k+1}^{-}-s_k^+<2\ln (\tau_{k+1}/\tau_k)<1.
\end{equation}
Like in the proof of Proposition \ref{SeqSheetProp}, it follows from the local monotonicity formula and changing variables that
\begin{equation} \label{DecayMassDropEqn}
\lim_{\bar{s}\to +\infty} \int_{\bar{s}}^{+\infty}\int_{M_s} \left\vert\mathbf{H}_{M_s}+\frac{\xX^\perp}{2}\right\vert^2 \, \phi^2_{(\yY,s)} \, \mathrm{e}^{-\frac{|\xX|^2}{4}} \, d\mathcal{H}^2 ds=0,
\end{equation}
where $M_s=M-\mathrm{e}^{s/2}\yY$ and $\phi_{(\yY,s)}(\xX)=\phi_\yY(\yY+\mathrm{e}^{-s/2}\xX)$ for $\phi_\yY$ any given cutoff function about $\yY$. In particular, we may assume that $\phi_{\yY}\equiv 1$ in $B^3_{|\yY|/2}(\yY)$. Thus, for each $s\in [s_k^+,s_{k+1}^{-}]$,
\begin{equation} \label{SupportEqn}
M^{i,k}_s\cap {\rm spt}(\psi)\subset M_s\cap\phi_{(\yY,s)}^{-1}(1).
\end{equation}

Hence, \eqref{MassDropEqn} together with \eqref{TimeEqn}, \eqref{DecayMassDropEqn}, and \eqref{SupportEqn} implies that
\begin{equation} \label{SmallMassDropEqn}
\left\vert\int_{M^{i,k}_{s^+_k}} \psi^2 \, \mathrm{e}^{-\frac{|\xX|^2}{4}} \, d\mathcal{H}^2-\int_{M^{i,k}_{s^{-}_{k+1}}} \psi^2 \, \mathrm{e}^{-\frac{|\xX|^2}{4}} \, d\mathcal{H}^2\right\vert
<C^\prime(M,i),
\end{equation}
where $C^\prime(M,i)\to 0$ as $i\to +\infty$. Therefore the claim follows from \eqref{SmallMassDropEqn} and Proposition \ref{SeqSheetProp}.

Next we fix any $i$ large enough so that the preceding claim holds. Suppose that, for some $k_0\geq i$ and some connected component $M^{i,k_0}$, 
$$
\partial M^{i,k_0}\cap\set{\xX\in\Real^3\colon\xX\cdot\yY=\tau_{k_0+1}|\yY|^2}=\bigcup_{l=1}^m \gamma^{i,k_0+1}_{j_l},
$$
where $m\geq 2$. Pick points $\zZ_l\in\gamma^{i,k_0+1}_{j_l}$ for $l\in\set{1,2}$. Let $\gamma_0$ be a smooth embedded curve in $M^{i,k_0}$ joining $\zZ_1$ and $\zZ_2$. By the preceding claim, for $l\in\set{1,2}$, there is a smooth embedded curve
$$
\gamma_l \subset M \cap \left(I_{\yY}(\tau_{k_0+1}|\yY|,+\infty)\times B_{R_i}^2\right)
$$
emanating from $\zZ_l$ and extending to infinity, and $\bar{\gamma}_1\cap\bar{\gamma}_2=\emptyset$. By a standard smoothing process and Proposition \ref{SeqSheetProp}, we may assume that $\gamma=\bar{\gamma}_0\cup\bar{\gamma}_1\cup\bar{\gamma}_2$ is a smooth embedded curve and that $\gamma$ and $\gamma^{i,k_0+1}_{j_l}$ intersect transversally only at $\zZ_l$ for $l\in\set{1,2}$. From the curve $\gamma$ we will construct, in the following paragraphs, a simple closed curve in $M$ that does not separate $M$. However, as $M$ is an end of self-shrinker in $\Real^3$ of finite topology, we may assume that $\bar{M}$ is homeomorphic to $\bar{B}^2_1\setminus\set{\oO}$. This leads to a contradiction. Therefore, for all $k\geq i$, each connected component of $M\cap B^3_{R_i}(\tau_k\yY)$ is joint, via the surface $M\cap (I_{\yY}(\tau_{k}|\yY|,\tau_{k+1}|\yY|)\times B_{R_i}^2)$, with exactly one of the connected components of $M\cap B^3_{R_i}(\tau_{k+1}\yY)$, and vice versa. This implies the claim in the beginning of the proof, concluding the theorem.

If $\Sigma^{\tau_{k_0+1}}=\Real_\yY\times\mathbb{S}^1$, then it follows from Proposition \ref{SeqSheetProp} that $\gamma^{i,k_0+1}_{j_1}$ is a simple closed smooth curve. By our construction, $\gamma$ and $\gamma^{i,k_0+1}_{j_1}$ intersect transversally only at $\zZ_1$. As $\bar{M}$ is connected at infinity, there exist points $\zZ^\prime_1\in\gamma_1$ and $\zZ^\prime_2\in\gamma_2$ so that they can be joint by a smooth embedded curve $\gamma^\prime$ in $M$ outside a large enough ball so that $\gamma^\prime\cap\gamma^{i,k_0+1}_{j_1}=\emptyset$. Denote by $\gamma^{\prime\prime}$ the part of $\gamma$ joining $\zZ^\prime_1$ and $\zZ^\prime_2$ that contains $\gamma_0$. Thus, again by a smoothing process, $\overline{\gamma^\prime\cup\gamma^{\prime\prime}}$ can be made a simple closed smooth curve in $M$ that transversally intersects $\gamma^{i,k_0+1}_{j_1}$ only at $\zZ_1$. This implies that $\overline{\gamma^\prime\cup\gamma^{\prime\prime}}$ does not separate $M$.

If $\Sigma^{\tau_{k_0+1}}=\Real_\yY\times\Real$, then we define $\tilde{\gamma}^{i,k_0+1}_{j_1}$ to be
$$
\left(M^{\tau_{k_0+1}}_{j_1}+\tau_{k_0+1}\yY\right)\cap B^3_{\frac{3R_i}{2}}(\tau_{k_0+1}\yY)\cap\set{\xX\in\Real^3 \colon\xX\cdot\yY=\tau_{k_0+1}|\yY|^2}.
$$
Let $\tilde{\zZ}_1$ and $\tilde{\zZ}_2$ be the ending points of $\tilde{\gamma}^{i,k_0+1}_{j_1}$. Thanks to Thom's transversality theorem, as we may assume $\bar{M}\subset\Real^3\setminus\bar{B}^3_4$, it follows from Lemma \ref{MaxPrincipleLem} that for $l\in\set{1,2}$, there is a curve $\tilde{\gamma}_l$ in $M\cap (\Real_{\tilde{\zZ}^\prime_l}\times\bar{B}_2^2)$ joining $\tilde{\zZ}_l$ to $\partial M$, where $\tilde{\zZ}_l^\prime$ can be chosen arbitrarily close to $\tilde{\zZ}_l$. Note that $\Real_{\tilde{\zZ}^\prime_l}\times\bar{B}_2^2$ is disjoint from $I_\yY(\frac{3}{4}\tau_{k_0+1}|\yY|,+\infty)\times B^2_{R_i}$. However, $\gamma$ is contained in $I_\yY(\tau_{k_0}|\yY|,+\infty)\times B^2_{R_i}$.
As $\tau_{k_0+1}\tau_{k_0}^{-1}$ is very close to $1$, $\gamma$ is disjoint from the closure of $\tilde{\gamma}_l$ for $l\in\set{1,2}$. Thus the closure of $\tilde{\gamma}_1\cup\tilde{\gamma}^{i,k_0+1}_{j_1}\cup\tilde{\gamma}_2$ and $\gamma$ intersect transversally only at $\zZ_1$. Finally, by the same reasoning, one can modify $\gamma$ to become a closed simple curve $\tilde{\gamma}$ with the above intersection property. As $\partial M$ is connected by our assumption, henceforth, $\tilde{\gamma}$ does not separate $M$.
\end{proof}

\section{Proof of Theorem \ref{MainThm}}
In what follows we finish the proof of Theorem \ref{MainThm}. By Brakke's regularity theorem \cite{B} (cf. \cite{WhReg}) the key is to show that the positive integer $L$ given in Theorem \ref{SheetThm} always equals $1$. For this purpose we adapt, in a straightforward manner, an argument of White \cite[Section 7]{Ch} for mean curvature flow of curves to higher dimensions.
\begin{proof}[Proof of Theorem \ref{MainThm}]
Let $\mathcal{K}$ be the associated Brakke flow for $M$, and let $\Lambda$ be the asymptotic cone of $M$ (cf. Proposition \ref{HausdorffAsympProp}). Take any $\yY\in\Lambda$ fixed; that is, $\yY\in\Real^3\setminus\set{\oO}$ with $\Theta_{(\yY,0)}(\mathcal{K})\geq 1$. Apply Theorem \ref{SheetThm} to $M$ and $\yY$. Let $L, \tau_0, R(\tau)$, and $\set{\Sigma^\tau}_{\tau>\tau_0}$ be given in the theorem. We hope to show the positive integer $L$ equal to $1$. We divide the proof into two cases according to the geometric property of the collection $\set{\Sigma^\tau}_{\tau>\tau_0}$ of self-shrinkers in $\Real^3$. 

Suppose that the $\Sigma^\tau$ are all $\Real_\yY\times\mathbb{S}^1$. As $M$ is connected, it follows that $L=1$ and, moreover, $\Lambda=\set{\tau\yY\colon\tau>0}$. Thus item \eqref{Cylinder} of Theorem \ref{MainThm} holds with $\vV(M)=\yY$ by Brakke's regularity theorem \cite{B} (cf. \cite{WhReg}). 

Suppose that the $\Sigma^\tau$ are all of the form $\Real_\yY\times\Real$. We argue by contradiction as in \cite[Section 7]{Ch}. Suppose $L>1$. As $M$ is embedded we can relabel indices so that for each $\tau>\tau_0$, the functions $f^\tau_j$, $j\in\set{1,\ldots,L}$, given in Theorem \ref{SheetThm} are ordered in the following manner: $f^\tau_{1}<f^\tau_{2}<\ldots<f^\tau_L$.
Define, for $\tau>\tau_0$,
$$
\delta(\tau)=f^\tau_L(\oO)-f^\tau_1(\oO).
$$
By \eqref{C1EstEqn}, $\delta(\tau)\to 0$ as $\tau\to +\infty$. Thus we may choose a sequence $\bar{\tau}_m\nearrow +\infty$ so that 
\begin{equation} \label{DeltaEqn}
\delta(\bar{\tau}_m)=\sup_{\tau\geq\bar{\tau}_m} \delta(\tau).
\end{equation}
Thus it follows from item \eqref{TangentFlow} of Proposition \ref{WhiteProp}, Theorem \ref{SheetThm}, and \cite[Theorem 3.4]{EHInterior} that, after passing to a subsequence and relabeling, $\mathcal{K}^{(\yY,0),\bar{\tau}_m}\cap(\Real^3\times (-\infty,0))$ converges smoothly to $\mathcal{T}=\set{\nu_t}_{t<0}$ where each $\nu_t=L\mathcal{H}^2\lfloor\Pi$ for some fixed static plane $\Pi\subset\Real^3$. Hence there is $\Omega_1\subset\Omega_2\subset\cdots\subset\Pi$ and $I_1\subset I_2\subset\cdots\subset\Real$ with $\bigcup\Omega_m=\Pi$ and $\bigcup I_m=\Real$ such that for each m, 
$$
\mathcal{K}^{(\yY,0),\bar{\tau}_m}\cap(\Omega_m\times\Real\times I_m)=\set{\mathcal{H}^2\lfloor \bigcup_{j=1}^L {\rm graph}(v^m_j(\cdot,t))}_{t\in I_m}
$$
where $v^m_1<v^m_2<\cdots<v^m_L$ are small functions on $\Omega_m\times I_m$.
Without loss of generality we assume $\Pi$ to be the $x_1x_2$-plane. Then $(v^m_L-v^m_1)$ satisfies a uniformly parabolic equation which is a small perturbation of the heat equation:
$$
\partial_t (v^m_L-v^m_1)=\sum_{k,l=1}^2 (\delta_{kl}+a^m_{kl})\partial_{k}\partial_{l} (v^m_L-v^m_1)+\sum_{k=1}^2 b^m_k \partial_{k} (v^m_L-v^m_1)+c^m(v^m_L-v^m_1),
$$
where $\partial_k$ denotes to take the partial derivative with respect to $x_k$. Moreover, the coefficients $a^m_{kl},b^m_k, c^m$ all converge to $0$ as $m\to +\infty$.

We define
$$
w_m(\pP,t)=\frac{v^m_L(\pP,t)-v^m_1(\pP,t)}{v^m_L(\oO,-1)-v_1^m(\oO,-1)}.
$$
Let $\tau_{m,t}=(-t)^{-1/2}\bar{\tau}_m$ for $t\in I_m$. Note that
$$
(1+\epsilon_m)^{-1}(-t)^{1/2}\delta(\tau_{m,t})<w_m(\oO,t)<(1+\epsilon_m)(-t)^{1/2}\delta(\tau_{m,t}),
$$
and $\epsilon_m\to 0^+$. Thus it follows from our choice of $\bar{\tau}_m$ (cf. \eqref{DeltaEqn}) that 
$$
w_m(\oO,t)<(1+\epsilon_m)^2(-t)^{1/2}w_m(\oO,-1)
$$
for $-1\leq t<0$. This together with the parabolic Harnack inequality implies that $w_m$ is uniformly bounded on compact subsets of $\Real^2\times (-\infty,0)$.
Hence, after passing to a subsequence and relabeling, $w_m\to w$ in $C^\infty_{loc}(\Real^2\times (-\infty,0))$, and $w$ satisfies the heat equation on $\Real^2$.
Moreover, for $-1\le t<0$, we have
\begin{equation} \label{UpperDecayEqn}
w(\oO,t)\leq (-t)^{1/2}w(\oO,-1).
\end{equation}
However, we may choose $c>0$ so that 
$$
c\mathrm{e}^{-\lambda_0}\psi_0(\pP)\leq w(\pP,-1),
$$
where $\lambda_0$ and $\psi_0$ are the first eigenvalue and the corresponding (positive) eigenfunction, respectively, of the Laplacian on $B^2_1$ with Dirichlet boundary condition. As $c\mathrm{e}^{\lambda_0 t}\psi_0$ also solves the heat equation on $\Real^2$, the strong maximum principle for parabolic equations implies that $$
w(\oO,t)\geq c\mathrm{e}^{\lambda_0 t}\psi_0(\oO)
$$
for all $t\in [-1,0)$. This contradicts \eqref{UpperDecayEqn}, and thus, $L$ must equal $1$. Hence, by Brakke's regularity theorem \cite{B} (cf. \cite{WhReg}), the curvature of $\sqrt{-t}\, M$ for $t\in [-1,0)$ is uniformly bounded in some neighborhood of $\yY$ in $\Real^3$. 

Therefore, if, for all $\yY\in\Lambda$, the $\Sigma^\tau$ given in Theorem \ref{SheetThm} are all of the form $\Real_\yY\times\Real$, then $\sqrt{-t}\, M$ converges in $C^\infty_{loc}(\Real^3\setminus\set{\oO})$ to $\Lambda$ as $t\to 0^{-}$, and in particular, $\Lambda$ is a regular cone in $\Real^3$. This implies item \eqref{RegCone} of Theorem \ref{MainThm} with $\mathcal{C}(M)=\Lambda$.
\end{proof}

\appendix

\section{} 
For readers' convenience we give a recounting of \cite[Theorem 19]{Song} on the linear growth bound for the second fundamental form of an end of a noncompact self-shrinker in $\Real^3$ of finite topology.
\begin{thm} \label{LinearCurvThm}
Let $M$ be an end of a self-shrinker in $\Real^3$ of finite topology. There is a constant $C_1=C_1(M)>0$ so that 
$$
|A_M(\xX)|\le C_1\left(|\xX|+1\right)
$$
for all $\xX\in M$.
\end{thm}
\begin{proof}
We argue by contradiction. Suppose that there was a sequence of points $\xX_i\in M$ with $|\xX_i|\to +\infty$ so that
$$
|\xX_i|^{-1}|A_M(\xX_i)|\to +\infty.
$$
Without loss of generality we may assume that $\bar{M}$ is a subset of $\Real^3\setminus\bar{B}_4$ and is homeomorphic to $\bar{B}^2_1\setminus\set{\oO}$.

Letting $r_i=|\xX_i|^{-1}$, we define, on $\bar{B}^3_{r_i}(\xX_i)\cap M$,
$$
f_i(\xX)=\left(r_i-|\xX-\xX_i|\right)|A_M(\xX)|.
$$
Observe that $f_i\equiv0$ in $\partial B^3_{r_i}(\xX_i)\cap M$. Thus $f_i$ attains its maximum at some point $\zZ_i\in B^3_{r_i}(\xX_i)\cap M$. Let $\sigma_i=\frac{1}{2}(r_i-|\zZ_i-\xX_i|)$. By the triangle inequality, for any $\xX\in B^3_{\sigma_i}(\zZ_i)$,
$$
r_i-|\xX-\xX_i| \geq \left(r_i-|\zZ_i-\xX_i|\right)-|\xX-\zZ_i|>\sigma_i.
$$
Thus it follows that
\begin{equation} \label{DoubleCurvEqn}
\sup_{\xX\in B^3_{\sigma_i}(\zZ_i)\cap M} |A_M (\xX)| \leq \sup_{\xX\in B^3_{\sigma_i}(\zZ_i)\cap M} \left(r_i-|\xX-\xX_i|\right)^{-1} f_i(\zZ_i)<2|A_M(\zZ_i)|.
\end{equation}

By the self-shrinker equation \eqref{ShrinkerEqn}, 
\begin{equation} \label{MeanCurvEqn}
\sup_{\xX\in B^3_{\sigma_i}(\zZ_i)\cap M} |H_M(\xX)|\leq |\zZ_i|.
\end{equation}
Thus, together with \eqref{MeanCurvEqn}, \cite[Theorem 1.3]{CZ}, and our assumption on the topology of $\bar{M}$, the local Gauss-Bonnet estimate of Ilmanen \cite[Theorem 3]{IlmanenSing} implies that
\begin{equation} \label{IntegralCurvEqn}
\int_{B^3_{\sigma_i}(\zZ_i)\cap M} |A_M|^2 \, d\mathcal{H}^2<C
\end{equation}
for some $C$ independent of $i$.

As $f_i(\zZ_i)\ge f_i(\xX_i)$, we have that
\begin{equation} \label{ScaleRadEqn}
\sigma_i |A_M(\zZ_i)|>\frac{1}{2} |\xX_i|^{-1}|A_M(\xX_i)|\to +\infty.
\end{equation}
Define the rescaled surfaces $N_i=|A_M(\zZ_i)|(\Sigma-\zZ_i)$. Thus $N_i$ satisfies
\begin{equation} \label{ScaleSurfaceEqn}
\mathbf{H}_{N_i}+\frac{1}{2}|A_M(\zZ_i)|^{-1}\left(\zZ_i+|A_M(\zZ_i)|^{-1}\xX\right)^\perp=\oO.
\end{equation}
Hence, together with \eqref{DoubleCurvEqn}, \eqref{IntegralCurvEqn}, \eqref{ScaleRadEqn}, and \eqref{ScaleSurfaceEqn}, it is implied from the Schauder estimates for uniformly elliptic equations and the Arzel\`{a}-Ascoli theorem that, after passing to a subsequence and relabeling, the $N_i$ converges locally smoothly (possibly with multiplicity) to a properly embedded minimal surface $N$ in $\Real^3$ of finite total curvature and genus $0$. Moreover, 
$$
|A_N(\oO)|=1 \mbox{ and } \sup_{\xX\in N}|A_N(\xX)|\le 2.
$$
Furthermore it follows from \cite{LR} that $N$ must be a Catenoid in $\Real^3$ which rotates about a straight line of direction $\mathbf{e}$.

Next we choose $R$ to be a sufficiently large constant so that $N\cap (B^3_{4R}\setminus B^3_R)$ is close to two parallel planes with normal $\mathbf{e}$. Let $M_i=|A_M(\zZ_i)|^{-1}\tilde{N}_i+\zZ_i$ where $\tilde{N}_i$ is the connected component of $N_i\cap B^3_{4R}$ containing $\oO$. For $i$ sufficiently large, $\tilde{N}_i$ is given by the normal exponential graph of a small function over some subset of $N$. Let $\gamma_i=|A_M(\zZ_i)|^{-1}\tilde{\gamma}_i+\zZ_i$ where $\tilde{\gamma}_i$ is the image of the closed geodesic, which encircles the neck of $N$, via the graphical representation of $\tilde{N}_i$. Thus $\gamma_i$ for $i$ sufficiently large is a simple closed planar curve in $M_i$. Denote by $M_i^+,M_i^{-}$ the two connected components of $M_i\setminus\gamma_i$.
By passing to a subsequence and relabeling, we may assume that $|\zZ_i|^{-1}\zZ_i\to\vV$. We divide the rest of the proof into two cases according to whether $\vV$ is parallel to $\mathbf{e}$ or not. 

Suppose that $\vV$ is parallel to $\mathbf{e}$.  By elementary geometry and Thom's transversality theorem, for each $i$ large, there is a solid cylinder $S_i=\Real_{\vV_i}\times\bar{B}_2^2$ so that $\partial S_i$ and $M\setminus\partial M$ intersect transversally, $S_i\cap M_i^\pm\neq\emptyset$, and $S_i\cap\gamma_i=\emptyset$.\footnote{Here is a recipe to find $S_i$. Let $\tilde{\zZ}$ be the center of the neck of $N$, and let $\tilde{\zZ}_i=|A_M(\zZ_i)|^{-1}\tilde{\zZ}+\zZ_i$. Let $\zZ_i^\prime$ be a point in the plane through $\tilde{\zZ}_i$ normal to $\mathbf{e}$ so that $2|A_M(\zZ_i)|^{-1}R<|\zZ^\prime_i-\tilde{\zZ}_i|<3|A_M(\zZ_i)|^{-1}R$. By continuity, there is a plane $\Pi_i$ through $\zZ_i^\prime$ such that $\mathbf{e}\times(\zZ_i^\prime-\tilde{\zZ}_i)$ is parallel to $\Pi$ and ${\rm dist}(\oO,\Pi_i)=2$. Observe that $\zZ^\prime_i-\tilde{\zZ}_i$ is almost normal to $\Pi$. Thus, $\Pi\cap\gamma_i=\emptyset$, and $\Pi$ and $M_i^\pm$ intersect nontrivially and transversally. Now let $S_i$ be the solid round cylinder of radius $2$ so that $\partial S_i$  is tangent to $\Pi$ at $\zZ_i^\prime$, and $S_i$ and $\gamma_i$ lie in different sides of $\Pi$. Finally we can adjust $\zZ^\prime_i$ in an arbitrarily small manner so that $\partial S_i$ and $M\setminus\partial M$  intersect transversally.} Furthermore, as $\bar{M}\subset\Real^3\setminus\bar{B}^3_4$, Lemma \ref{MaxPrincipleLem} implies that $M_i^\pm$ can be connected to $\partial M$ via the surface $\bar{M}\cap S_i$. Hence it follows that $\gamma_i$ does not separate $\bar{M}$, leading to a contradiction.

Suppose that $\vV$ is not parallel to $\mathbf{e}$. Then, for $i$ sufficiently large,
$$
\inf_{\xX\in M_i^\pm} |\xX|<\inf_{\xX\in\gamma_i} |\xX|=a_i.
$$
However, by \cite[Lemma 3.20]{CMGenSing},
$$
\Delta_M |\xX|^2-\frac{1}{2}\xX\cdot\nabla_M |\xX|^2+|\xX|^2-4=0.
$$
As $\bar{M}\subset\Real^3\setminus\bar{B}_4^3$, it follows from the maximum principle for elliptic equations that $M_i^\pm$ can be connected to $\partial M$ via the surface $\bar{M}\cap B^3_{a_i}$. This implies that $\gamma_i$ does not separate $\bar{M}$, giving a contradiction.
\end{proof}

\begin{acknowledgement*}
The author is very grateful to Otis Chodosh for generously sharing his lecture notes on mean curvature flow based on a course taught by Brian White. She would also like to thank Otis Chodosh, Tom Ilmanen, Rob Kusner, Nan Li, and Neshan Wickramasekera for helpful discussions. Finally this work was initiated during the author's visit of Cambridge University in May and June 2013, and she thanks Neshan Wickramasekera for his kind invitation and hospitality.
\end{acknowledgement*}


\begin{thebibliography}{99}
\bibitem{AL} U. Abresch and J. Langer, \emph{The normalized curve shortening flow and homothetic solutions}, J. Differential Geom. 23 (1986), 175--196.

\bibitem{A} S. Angenent, \emph{Shrinking doughnuts}, in Nonlinear Diffusion Equations and Their Equilibrium States, 3 (Gregynog, 1989), 21--38, Progr. Nonlinear Differential Equations Appl., 7, Birkh\"{a}user Boston, Boston, MA, 1992.

\bibitem{BWHM} J. Bernstein and L. Wang, \emph{A remark on a uniqueness property of high multiplicity tangent flows in dimension three}, Int. Math. Res. Not. IMRN 2015 (2015), no. 15, 6286--6294.

\bibitem{BW} J. Bernstein and L. Wang, \emph{A sharp lower bound for the entropy of closed hypersurfaces up to dimension six}, Invent. Math., to appear. Available at: \url{http://arxiv.org/abs/1406.2966}.

\bibitem{B} K. Brakke, The motion of a surface by its mean curvature, Mathematical Notes 20, Princeton University Press, Princeton, N.J., 1978.

\bibitem{Be} S. Brendle, \emph{Embedded self-similar shrinkers of genus $0$}, Ann. of Math. (2) 183 (2016), no. 2, 715--728.

\bibitem{CS} P.-Y. Chang and J. Spruck, \emph{Self-shrinkers to the mean curvature flow asymptotic to isoparametric cones}. Preprint. Available at: \url{http://arxiv.org/abs/1510.07183}.

\bibitem{CZ} X. Cheng and D. Zhou, \textit{Volume estimate about shrinkers},  Proc. Amer. Math. Soc. 141 (2013), no. 2, 687--696.

\bibitem{Ch} O. Chodosh, \emph{Brian White-Mean curvature flow (Math 258) lecture notes}. Unpublished notes.

\bibitem{CS} H.I. Choi and R. Schoen, \emph{The space of minimal embeddings of a surface into a three-dimensional manifold of positive Ricci curvature}, Invent. Math. 81 (1985), no. 3, 387--394.

\bibitem{CIM} T.H. Colding, T. Ilmanen, and W.P. Minicozzi II, \emph{Rigidity for generic singularities of mean curvature flow}, Publ. Math. Inst. Hautes Etudes Sci. 121 (2015), 363--382.

\bibitem{CMGrad} T.H. Colding and W.P. Minicozzi II, \emph{Sharp estimates for mean curvature flow of graphs}, J. Reine Angew. Math. 574 (2004), 187--195.

\bibitem{CMMS1} T.H. Colding and W.P. Minicozzi II, \emph{The space of embedded minimal surfaces of fixed genus in a 3-manifold. I. Estimates off the axis for disks}, Ann. of Math. (2) 160 (2004), no. 1, 27--68.

\bibitem{CMMS2} T.H. Colding and W.P. Minicozzi II, \emph{The space of embedded minimal surfaces of fixed genus in a 3-manifold. II. Multi-valued graphs in disks}, Ann. of Math. (2) 160 (2004), no. 1, 69--92.

\bibitem{CMMS3} T.H. Colding and W.P. Minicozzi II, \emph{The space of embedded minimal surfaces of fixed genus in a 3-manifold. III. Planar domains}, Ann. of Math. (2) 160 (2004), no. 2, 523--572.

\bibitem{CMMS4} T.H. Colding and W.P. Minicozzi II, \emph{The space of embedded minimal surfaces of fixed genus in a 3-manifold. IV. Locally simply connected}, Ann. of Math. (2) 160 (2004), no. 2, 573--615.

\bibitem{CMGenSing} T.H. Colding and W.P. Minicozzi II, \emph{Generic mean curvature flow I: generic singularities},  Ann. of Math. (2) 175 (2012), no. 2, 755--833.

\bibitem{DX} Q. Ding and Y.-L. Xin, \emph{olume growth, eigenvalue and compactness for self-shrinkers}, Asian J. Math. 17 (2013), no. 3, 443--456.

\bibitem{Ec} K. Ecker, \emph{On regularity for mean curvature flow of hypersurfaces}, Calc. Var. Partial Differential Equations 3 (1995), no. 1, 107--126.

\bibitem{EHGraph} K. Ecker and G. Huisken, \emph{Mean curvature evolution of entire graphs}, Ann. of Math. (2), 130 (1989), no. 3, 453--471.

\bibitem{EHInterior} K. Esker and G. Huisken, \emph{Interior estimates for hypersurfaces moving by mean curvature}, Invent. Math. 105 (1991), no. 3, 547--569.

\bibitem{HK} R. Haslhofer and B. Kleiner, \emph{Mean curvature flow of mean convex hypersurfaces}, Comm. Pure Appl. Math., to appear. Available at: \url{http://arxiv.org/abs/1304.0926}.

\bibitem{H1} G. Huisken, \textit{Asymptotic behavior for singularities of the mean curvature flow}, J. Differential Geom. 31 (1990), no. 1, 285--299.

\bibitem{H2} G. Huisken, \emph{Local and global behaviour of hypersurfaces moving by mean curvature}, in Differential Geometry: Partial Differential Equations on Manifolds (Los Angeles, CA, 1990), 175--191, Proc. Sympos. Pure Math., 54, Part 1, Amer. Math. Soc., Providence, RI, 1993.

\bibitem{IlmanenElliptic} T. Ilmanen, Elliptic regularization and partial regularity for motion by mean curvature, Mem. Amer. Math. Soc. 108 (1994), no. 520.

\bibitem{IlmanenSing} T. Ilmanen, \emph{Singularities of mean curvature flow of surfaces}. Preprint. Available at: \url{https://people.math.ethz.ch/~ilmanen/papers/sing.ps}.

\bibitem{IlmanenLec} T. Ilmanen, Lectures on mean curvature flow and related equations. Unpublished notes. Available at: \url{http://www.math.ethz.ch/~ilmanen/papers/pub.html}. 

\bibitem{KKM} N. Kapouleas, S.J. Kleene and N.M. M\o{}ller, \emph{Mean curvature self-shrinkers of high genus: non-compact examples}, J. Reine Angew. Math., to appear. Available at: \url{http://arxiv.org/abs/1106.5454}.

\bibitem{K} D. Ketover, \emph{Self-shrinking Platonic solids}. Preprint. Available at: \url{http://arxiv.org/abs/1602.07271}.

\bibitem{LW} H. Li and B. Wang, \emph{The extension problem for the mean curvature flow (I)}. Preprint. Available at: \url{http://arxiv.org/abs/1608.02832}.

\bibitem{LR} F.J. L\'{o}pez and A. Ros, \emph{On embedded complete minimal surfaces of genus zero}, J. Differential Geom. 33 (1991), no. 1, 293--300. 

\bibitem{MR} W. Meeks III and H. Rosenberg, \emph{The uniqueness of the helicoid}, Ann. of Math. (2) 161 (2005), no. 2, 727--758.

\bibitem{N} X.H. Nguyen, \emph{Construction of complete embedded self-similar surfaces under mean curvature flow. Part III.}, Duke Math. J. 163, no. 11 (2014), 2023--2056.

\bibitem{Song}  A. Song, \emph{A maximum principle for self-shrinkers and some consequences}. Preprint. Available at: \url{http://arxiv.org/abs/1412.4755}.

\bibitem{WaBern} L. Wang, \emph{A Bernstein type theorem for self-similar shrinkers}, Geom. Dedicata 151 (2011), no. 1, 297--303. 

\bibitem{WaCone} L. Wang, \emph{Uniqueness of self-similar shrinkers with asymptotically conical ends}, J. Amer. Math. Soc. 27 (2014), no. 3, 613--638.

\bibitem{WaCyl} L. Wang, \emph{Uniqueness of self-similar shrinkers with asymptotically cylindrical ends}, J. Reine Angew. Math. 715 (2016), 207--230. 

\bibitem{WhStrata} B. White, \emph{Stratification of minimal surfaces, mean curvature flows, and harmonic maps}, J. Reine Angew. Math. 488 (1997), 1--35.

\bibitem{WhSize} B. White, \emph{The size of the singular set in mean curvature flow of mean convex sets}, J. Amer. Math. Soc. 13 (2000), no. 3, 665--695.

\bibitem{WhReg} B. White, \emph{A local regularity theorem for mean curvature flow}, Ann. of Math. (2) 161 (2005), no. 3, 1487--1519.
\end{thebibliography}
\end{document}